\documentclass{article}

\usepackage{arxiv}
\usepackage[applemac]{inputenc} 		% allow utf-8 input
\usepackage[T1]{fontenc}    		% use 8-bit T1 fonts
\usepackage[colorlinks]{hyperref}      % hyperlinks
\usepackage{url}            			% simple URL typesetting
\usepackage{amsthm}
\usepackage{booktabs}       		% professional-quality tables
\usepackage{amsfonts}       		% blackboard math symbols
\usepackage{amsmath}

\usepackage{nicefrac}       		% compact symbols for 1/2, etc.
\usepackage{microtype}      		% microtypography
\usepackage{lipsum}				% Can be removed after putting your text content
\usepackage[square,numbers]{natbib}
\usepackage{mathtools}
\usepackage{algorithm}
\usepackage{algorithmicx}
\usepackage{algpseudocode}
\usepackage{graphicx}
\usepackage{indentfirst,latexsym,bm}
\usepackage{amsmath}
\usepackage{subfigure}
\usepackage{amssymb}
\usepackage{xcolor}
\usepackage{comment}
\usepackage{enumitem}
\usepackage{bbm}
\usepackage{tikz}
\usepackage{mdframed}

\usetikzlibrary{arrows.meta,positioning}
\usepackage{pgfplots}
\pgfplotsset{compat=newest}
\usepgfplotslibrary{fillbetween}
\usetikzlibrary{shapes,decorations}
\usetikzlibrary{fit}

\colorlet{color1}{blue}
\colorlet{color2}{red!50!black}

\definecolor{ivory}{RGB}{218,215,203}

\definecolor{cuhkp}{RGB}{98,56,105} 	% purple dark
\definecolor{cuhkpl}{RGB}{152,24,147} 	% purple light
\definecolor{cuhkb}{RGB}{219,160,1} 	% ocher
\definecolor{cuhkbd}{RGB}{178,129,0} 	% ocher dark
\definecolor{cuhkr}{RGB}{88,35,155}  	% magenta-red
\definecolor{blackp}{RGB}{0,0,0} 
\definecolor{redp}{RGB}{255,0,0}
\definecolor{orangep}{RGB}{255,128,0}
\definecolor{brownp}{RGB}{128,77,0}
\definecolor{yellowp}{RGB}{255,230,0}
\definecolor{greenp}{RGB}{128,230,0}
\definecolor{bluep}{RGB}{0,128,255}
\definecolor{purplep}{RGB}{152,24,147}
\definecolor{pinkp}{RGB}{230,0,128}    

\usepackage{hyperref}[6.83]

\hypersetup{
    colorlinks=true, % Enable colored links instead of boxes
    linkcolor=blue,  % Color of internal links
    citecolor=blue, % Color of citation links (e.g., \cite{})
    urlcolor=magenta % Color of URL links
}

\RequirePackage[capitalize,nameinlink]{cleveref}

\crefformat{equation}{\textup{#2(#1)#3}}
\crefrangeformat{equation}{\textup{#3(#1)#4--#5(#2)#6}}
\crefmultiformat{equation}{\textup{#2(#1)#3}}{ and \textup{#2(#1)#3}}
{, \textup{#2(#1)#3}}{, and \textup{#2(#1)#3}}
\crefrangemultiformat{equation}{\textup{#3(#1)#4--#5(#2)#6}}%
{ and \textup{#3(#1)#4--#5(#2)#6}}{, \textup{#3(#1)#4--#5(#2)#6}}{, and \textup{#3(#1)#4--#5(#2)#6}}

\Crefformat{equation}{#2Equation~\textup{(#1)}#3}
\Crefrangeformat{equation}{Equations~\textup{#3(#1)#4--#5(#2)#6}}
\Crefmultiformat{equation}{Equations~\textup{#2(#1)#3}}{ and \textup{#2(#1)#3}}
{, \textup{#2(#1)#3}}{, and \textup{#2(#1)#3}}
\Crefrangemultiformat{equation}{Equations~\textup{#3(#1)#4--#5(#2)#6}}%
{ and \textup{#3(#1)#4--#5(#2)#6}}{, \textup{#3(#1)#4--#5(#2)#6}}{, and \textup{#3(#1)#4--#5(#2)#6}}

\crefdefaultlabelformat{#2\textup{#1}#3}

\theoremstyle{plain}
\newtheorem{thm}{Theorem}%[section]
\newtheorem{lemma}{Lemma}[section]
\newtheorem{lem}[thm]{Lemma}

\newtheorem{proposition}[thm]{Proposition}
\newtheorem{remark}[thm]{Remark}
\newtheorem{assumption}{Assumption}

\crefformat{assumption}{#2Assumption~#1#3}
\Crefformat{assumption}{#2Assumption~#1#3}

\theoremstyle{plain}

\newcommand{\rmn}[1]{\textup{\textrm{#1}}}

\newcommand{\SGD}{\mathsf{SGD}}

\newcommand{\DGD}{\mathsf{DGD}}
\newcommand{\RR}{\mathsf{RR}}

\newcommand{\Fed}{\mathsf{FedAvg}}

\newcommand{\one}{\mathbf{1}}

\newcommand{\cO}{\mathcal{O}}

\newcommand{\cB}{\mathcal{B}}

\newcommand{\R}{\mathbb{R}}
\newcommand{\sL}{{\sf L}}
\newcommand{\sG}{{\sf G}}
\newcommand{\sT}{{\sf T}}
\newcommand{\sC}{{\sf C}}
\newcommand{\sH}{{\sf H}}

\newcommand{\bx}{\mathbf{x}}

\newcommand{\N}{\mathbb{N}}
\newcommand{\Rn}{\mathbb{R}^d}

\newcommand{\Prob}{\mathbb{P}}
\newcommand{\Exp}{\mathbb{E}}

\newcommand{\dist}{\mathrm{dist}}
\newcommand{\crit}{\mathrm{crit}}
\newcommand{\acc}{\mathcal{A}}

\newcommand{\cS}{\mathcal{S}}

\newcommand{\half}{\frac{1}{2}}
\newcommand{\iprod}[2]{\left\langle #1, #2 \right\rangle}

\newcommand{\be}{\begin{equation}}
\newcommand{\ee}{\end{equation}}

\newcommand{\spa}{\mathrm{span}}
\newcommand{\nul}{\mathrm{null}}

\newcommand{\sct}[1]{\boldsymbol{#1}}
\newcommand{\se}{\sct{e}}
\newcommand{\sg}{\sct{g}}

\newcommand{\scs}{\sct{s}}

\newcommand{\sx}{\sct{x}}

\newcommand{\sK}{\sct{K}}

\title{A KL-based Analysis Framework with Applications to Non-Descent Optimization Methods}

\author{
Junwen Qiu\thanks{Contributed equally to this work.} \\
	School of Data Science (SDS) \\
	The Chinese University of Hong Kong, Shenzhen \\ 
	Shenzhen, Guangdong, China \\
	\texttt{junwenqiu@link.cuhk.edu.cn} \\
        \And
	  Bohao Ma\footnotemark[1] \\ 
	School of Data Science (SDS) \\
	The Chinese University of Hong Kong, Shenzhen \\
        Shenzhen Research Institute of Big Data (SRIBD) \\
	Shenzhen, Guangdong, China \\
	\texttt{bohaoma@link.cuhk.edu.cn} \\
        \And
	Xiao Li \\ 
	School of Data Science (SDS) \\
	The Chinese University of Hong Kong, Shenzhen \\
	Shenzhen, Guangdong, China \\
	\texttt{lixiao@cuhk.edu.cn} \\
	\And
	Andre Milzarek\\
	School of Data Science (SDS) \\
	The Chinese University of Hong Kong, Shenzhen \\ 
	Shenzhen, Guangdong, China \\
	\texttt{andremilzarek@cuhk.edu.cn} \\
}

\begin{document}
\maketitle

\begin{abstract}
We propose a novel analysis framework for non-descent-type optimization methodologies in nonconvex scenarios based on the Kurdyka-{\L}ojasiewicz property. Our framework allows covering a broad class of algorithms, including those commonly employed in stochastic and distributed optimization. Specifically, it enables the analysis of first-order methods that lack a sufficient descent property and do not require access to full (deterministic) gradient information. We leverage this framework to establish, for the first time, iterate convergence and the corresponding rates for the decentralized gradient method and federated averaging under mild assumptions. Furthermore, based on the new analysis techniques, we show the convergence of the random reshuffling and stochastic gradient descent method without necessitating typical a priori bounded iterates assumptions. 
\end{abstract}

\section{Introduction}
The analysis of algorithms under the Kurdyka-{\L}ojasiewicz (KL) inequality \cite{lojasiewicz1965ensembles,kur98} has become an active and fruitful area of research in nonconvex optimization for several compelling reasons. Firstly, the KL inequality is a local geometric property that is satisfied for a vast and ubiquitous class of functions \cite{AttBolRedSou10,AttBolSva13,li2018calculus,liu2019quadratic}. Notably, in \cite{absil2005convergence,AttBolRedSou10,AttBolSva13,BolSabTeb14}, Absil, Attouch and Bolte et al. have made significant contributions in developing a comprehensive KL analysis framework for optimization algorithms. This framework serves as a valuable blueprint, simplifying the study and verification of the asymptotic behavior and iterate convergence of descent-type methods for nonsmooth nonconvex minimization problems. Successful instances of applying the analysis framework to obtain convergence encompass splitting methods \cite{li2015global}, Douglas-Rachford splitting \cite{li2016douglas}, proximal gradient method \cite{AttBolSva13} and its inertial variant \cite{OchCheBroPoc14}, among many others. Moreover, when the {\L}ojasiewicz exponent of the objective function is known, local rates of convergence can be quantified and explicitly derived, see, e.g., \cite{AttBol09,AttBolRedSou10,BolSabTeb14,fragarpey15,li2018calculus}.

The large-scale nature of modern machine learning problems has made stochastic and distributed optimization techniques crucial tools in manifold tasks, including the training of neural networks. Nonetheless, the asymptotic convergence analysis of these algorithms in the nonconvex setting is relatively limited compared to deterministic methods. This limitation is primarily due to the challenges posed by the lack of a sufficient descent property. Additionally, stochastic and distributed algorithms do not have access to the full gradient information, necessitating the use of diminishing step sizes to mitigate the stochastic errors. This significantly differs from the common step size conditions in the standard frameworks \cite{absil2005convergence,AttBolSva13,fragarpey15,OchCheBroPoc14}, where the step sizes are typically bounded away from zero.

Given these observations in stochastic and distributed optimization, we introduce a novel KL-based analysis framework outlined (in an informal way) below. 
\begin{mdframed}[
  linecolor=black!100,%red!50!black,      % Border color
  linewidth=.2mm,          % Border width
  backgroundcolor=bluep!1,% Background color
  roundcorner=0pt,       % Corner radius
  innerleftmargin=10pt,   % Left margin
  innerrightmargin=15pt,  % Right margin
  innertopmargin=10pt,    % Top margin
  innerbottommargin=10pt  % Bottom margin
]

\underline{\textbf{Informal Framework}}. 
Let $K\in\N$, $\{x^k\}_k \subset \Rn$, $\{\beta_k\}_k \subset \R$, and $\{I_k\}_k \subseteq \N$ be given. We assume: 
%\vspace{2mm}
\begin{enumerate}[label=\textup{\textrm{(A.\arabic*)}},topsep=0pt, leftmargin = 28pt,itemsep=0ex,partopsep=0ex]
	\item  \label{A1-informal} \textbf{Approximate descent.} There exist $a_1>0$ and $\{p_k\}\subseteq \R_+$ such that for any integer $k \geq K$,
	\[
		f({x}^{I_{k+1}}) +a_1 \beta_k \|\nabla{f}({x}^{I_k})\|^2 \leq f({x}^{I_k})  +  p_k.
\]	
	\item \label{A2-informal} \textbf{Gradient-bounded update.} There exist $a_2>0$ and $\{q_k\}\subseteq \R_+$ such that for any integer $k \geq K$,
\[
		\max_{I_k < i \leq I_{k+1}}\|x^{i}-x^{I_k}\| \leq a_2 \beta_k\|\nabla f(x^{I_k})\| + q_k. \vspace{-1ex}
\]
\end{enumerate}
\end{mdframed}
The sequence $\{\beta_k\}_k$ is closely related to the choice of step sizes and can converge to zero. Condition \ref{A1-informal} captures an approximate descent-type property, where $\{p_k\}_k$ is associated with the potential \emph{non-descent} terms. This allows covering algorithms that might not strictly adhere to a descent condition at each iteration.  Condition \ref{A2-informal} reflects the iterative structure of many first-order methods in the smooth case, where updates are gradient-based but potentially incorporate \emph{approximation errors} $\{q_k\}_k$. Approximation errors and non-descent behavior frequently arise in stochastic and distributed settings. Notably, the non-descent terms $\{p_k\}_k$ and errors $\{q_k\}_k$ often exhibit a dependence on higher orders of the step size, as observed in algorithms for finite-sum optimization, see \cref{sec:finite sum}. Furthermore, the sequence $\{I_k\}_k$ enables us to depart from the traditional consecutive-iterate analyses. This is particularly useful when characterizing the almost sure behavior for stochastic approximation methods, see \cref{sec:sgd} and \cite[Section 3]{qiu2024convergence}. % for more explanations.

In summary, we can leverage this framework to analyze algorithms that possess the following characteristics:
\begin{itemize}
	\item[--] Do not necessarily exhibit \emph{sufficient descent} property.
	\item[--] Could be working with \emph{inexact gradient} information. 
\end{itemize}
We now present an informal overview of our main convergence theorem.

\begin{mdframed}[
  linecolor=black!100,%red!50!black,      % Border color
  linewidth=.2mm,          % Border width
  backgroundcolor=bluep!1,% Background color
  roundcorner=0pt,       % Corner radius
  innerleftmargin=10pt,   % Left margin
  innerrightmargin=15pt,  % Right margin
  innertopmargin=10pt,    % Top margin
  innerbottommargin=10pt  % Bottom margin
]
\underline{\textbf{\cref{thm:main} (Informal)}}. \vspace{-1ex}
\begin{itemize}[leftmargin=5ex]
    \item [--] Let $f:\Rn\to\R$ be a Lipschitz smooth function and let the KL property hold at $\crit(f)$.
    \item [--]  Suppose $\{x^k\}_k$ satisfies \ref{A1-informal}--\ref{A2-informal} and let certain summability conditons hold for $\{p_k\}_k,\{q_k\}_k$. \vspace{-1ex}
\end{itemize}
Then, %$\|\nabla f(x^k)\|\to0$, $\{f(x^k)\}_k$ converges, and 
either $\{x^k\}_k$ converges to some stationary point $x^*\in\crit(f)$ or it holds that $\|x^k\| \to \infty$. \vspace{-1ex}
\end{mdframed}
This theorem shows iterate convergence for $\{x^k\}_k$, i.e., $\{x^k\}_k$ converges to some stationary point of $f$ when $\|x^k\|\not\rightarrow\infty$. Note that $\|x^k\| \to \infty$ happens if and only if $\{x^k\}_k$ does not have any accumulation point. This can be excluded if $\{x^k\}_k$ is bounded \cite{AttBol09,BolSabTeb14,li2023convergence,qian2023convergence} or the function $f$ is coercive \cite{daneshmand2020second}. Some works also assume the existence of at least an accumulation point for $\{x^k\}_k$, e.g., \cite{AttBolSva13,OchCheBroPoc14,li2015global,li2016douglas}. Beyond the main convergence results, we also quantify the local convergence rates for $\{x^k\}_k$ in \cref{thm:convergence rate} when $\{p_k\}_k$ and $\{q_k\}_k$ admit a specific form.

\subsection{Related Work}

Attouch, Bolte and Svaiter \cite{AttBolSva13} have developed a comprehensive KL inequality-based analysis framework (outlined below) that unifies and simplifies the convergence analysis of descent-type methods applied to nonconvex problems. 

\begin{mdframed}[
  linecolor=black!100,%red!50!black,      % Border color
  linewidth=.2mm,          % Border width
  backgroundcolor=bluep!1,% Background color
  roundcorner=0pt,       % Corner radius
  innerleftmargin=10pt,   % Left margin
  innerrightmargin=15pt,  % Right margin
  innertopmargin=10pt,    % Top margin
  innerbottommargin=10pt  % Bottom margin
]
Let $f:\Rn\to\R$ satisfies the KL property at $\crit(f)$. If $\{x^k\}_k$ has at least one accumulation point and satisfies \ref{H1}--\ref{H2} below, then $\{x^k\}_k$ converges to some $x^*\in\crit(f)$.
\begin{enumerate}[label=\textup{\textrm{(H.\arabic*)}},topsep=0pt,itemsep=0ex,partopsep=0ex]
\item \label{H1} \textbf{Sufficient descent.} There exists a constant $a>0$, such that for all $k\in\N$,
\begin{equation}
	\label{eq:H1}
	f({x}^{k+1}) + a \|x^{k+1}-x^k\|^2 \leq f({x}^{k}).
\end{equation}
\item \label{H2} \textbf{Relative error.} There exist constants $b>0$, such that for all $k\in\N$,
\begin{equation}
	\label{eq:H2}
\dist(0,\partial f(x^{k+1})) \leq b 	\|x^{k+1}-x^k\|. \vspace{-1ex}
\end{equation}
\end{enumerate}
\end{mdframed}
Here, $\partial f(x)$ denotes the limiting subdifferential \cite[Definition 1]{AttBol09} of $f$ at $x\in\Rn$ and $\dist(0,S):=\inf_{s\in S}\|s\|$. \ref{H1} is intended to model a descent property as it involves a measure of the quality of the descent. \ref{H2} reflects relative inexact optimality conditions.

As mentioned before, the framework \cite{AttBolSva13} has been successfully utilized to establish iterate convergence for a vast variety of algorithms. 
Nonetheless, \ref{H1} typically does not hold when analyzing inertial methods due to additional momentum terms. To fit inertial methods, Ochs et al. have studied adjustments of the conditions \ref{H1} and \ref{H2} in \cite[Section 3.2]{OchCheBroPoc14} and iterate convergence of inertial proximal methods was shown \cite[Theorem 4.9]{OchCheBroPoc14}.
Frankel, Garrigos and Peypouquet \cite{fragarpey15} noted that the condition \ref{H2} is restrictive when algorithmic updates contain small computational errors. To address this, a relaxation of condition \ref{H2} was proposed in \cite{fragarpey15} that allows incorporating additive errors:
\begin{enumerate}[label=\textup{\textrm{(H.2$^\prime$)}},topsep=0pt,itemsep=0ex,partopsep=0ex]
\item \label{H2p} \textbf{Relative error.} There exist constants $b>0$, such that for all $k \geq 1$,
\begin{equation}
	%\label{eq:H2}
\dist(0,\partial f(x^{k+1})) \leq b \|x^{k+1}-x^k\| + \varepsilon_k,\quad \text{where}\quad {\textstyle \sum_{k=1}^\infty}\; \varepsilon_k < \infty. 
\end{equation}
\end{enumerate}
This modification is non-trivial because it helps to cover the inexact proximal gradient method and the generalized Newton's method, cf. \cite[Theorem 6 and Proposition 4]{fragarpey15}. Moreover, explicit convergence rates in terms of the parameters (usually the KL exponent) are also provided in \cite[Theorem 4 and 5]{fragarpey15}. 

Still, the necessity for sufficient descent, \ref{H1}, in the existing frameworks \cite{AttBolSva13,OchCheBroPoc14,fragarpey15} has hindered the convergence analyses of algorithms without such property. Typical examples for such algorithms include methods with nonmonotone line-search schemes \cite{grippo1986nonmonotone,birgin2000nonmonotone,grippo2002nonmonotone,qian2023convergence}, to name a few. Taking the specific line search conditions (cf. \cite[Section 3]{grippo1986nonmonotone}) into account, Qian and Pan \cite{qian2023convergence} developed a new condition \ref{H1p} (shown below) and established the iterates convergence.
\begin{enumerate}[label=\textup{\textrm{(H.1$^\prime$)}},topsep=0pt,itemsep=0ex,partopsep=0ex]
\item \label{H1p} \textbf{Nonmonotone descent.} There exists a constant $a>0$, such that for all $k\geq 1$,
\begin{equation}
	\label{eq:H1p}
	f({x}^{k+1}) + a \|x^{k+1}-x^k\|^2 \leq {\max}_{j=[k-m]_{+},\ldots,k}\; f({x}^{j}).
\end{equation}
\end{enumerate}
The integer $m$ is a preset parameter in the line-search scheme and $[k-m]_+:=\max\{0,k-m\}$. The adjusted framework enables the convergence analysis of nonmonotone line-search methods, cf. \cite[Proposition 2.5 and Theorem 2.9]{qian2023convergence}. 

Incomplete gradient information during iterations is another factor leading to the lack of a sufficient descent property. Common approaches with this feature include stochastic methods such as stochastic gradient descent ($\SGD$), random reshuffling ($\RR$) \cite{li2023convergence}, and distributed algorithms such as decentralized gradient descent ($\DGD$) \cite{nedic2009distributed} and federated averaging ($\Fed$) \cite{mcmahan2017communication}. Unfortunately, the existing frameworks fail to provide convergence guarantees for these methods. The reasons behind such limitations are twofold: (i) the absence of sufficient descent properties and (ii) the complex dynamics arising from stochastic (inexact) gradients and diminishing step sizes. It is important to note that the step sizes in the existing frameworks \cite{AttBolSva13,OchCheBroPoc14,fragarpey15,qian2023convergence} are required to be bounded away from zero. This excludes the use of diminishing step sizes, which can help mitigate the errors from stochastic (inexact) gradients. To our knowledge, existing KL-based frameworks are not applicable to approaches that only possess approximate descent guarantees and utilize diminishing step sizes.

\subsection{Contribution and Organization}
We summarize our contributions below.

In \cref{subsec:new_framework}, we introduce an abstract framework that allows us to prove convergence of non-descent first-order methods in the smooth nonconvex setting. Notably, compared to existing frameworks, our proposed analysis framework enables the convergence analysis of stochastic and distributed methods. In \cref{subsec:framework+conv}, we provide the formal KL-based analysis framework and an abstract convergence result for sequences satisfying the framework (\cref{thm:main}). In \cref{subsec:special case}, we introduce a specialized framework for studying the convergence rates of non-descent methods with polynomial step sizes (\cref{thm:convergence rate}).
 
In \cref{sec:sgd} and \ref{sec:finite sum}, our proposed framework is applied to analyze various non-descent stochastic or distributed algorithms. First, we recover the convergence results of $\SGD$ and $\RR$ without requiring the typical a priori boundedness assumptions in \cref{sec:sgd} and \cref{sec:rr}, respectively. In \cref{sec:dgd} and \cref{sec:fed}, we conduct in-depth analyses for $\DGD$ and $\Fed$. To the best of our knowledge, the convergence results for $\DGD$ and $\Fed$ are new in the nonconvex setting.

Beyond convergence results, we further quantify the local convergence rates of $\DGD$, $\RR$, and $\Fed$ while employing polynomial step sizes. For $\RR$, our analysis extends a prior work \cite{li2023convergence} by providing the convergence rates for both function values and the norm of the gradients. Moreover, the convergence rates for $\DGD$ and $\Fed$ seem to be new in the context of nonconvex optimization.

\section{Framework and Convergence Results}\label{subsec:new_framework}
\subsection{Basic Assumptions}
Throughout this paper, we make the following assumptions on the objective function $f$.
 \begin{assumption}%[Objective function]
	\label[assumption]{as:func-1}
	The function $f$ is ${\sf L}$-smooth and is bounded from below by some $\bar f\in\R$.
\end{assumption}
This assumption is quite common in smooth optimization, see, e.g., \cite{AttBol09,OchCheBroPoc14,shi2015extra,yuan2016convergence,nesterov2018lectures,li2023convergence}. It also provides a useful bound \cite[Eq. (2.5)]{li2023convergence} for $\|\nabla f(x)\|$, i.e., 
\begin{equation}
	\label{eq:L-smooth}
	\|\nabla f(x)\|^2 \leq 2\sL  (f(x) - \bar f),\quad \forall~x\in\Rn.
\end{equation}
\begin{assumption}
	\label[assumption]{as:kl} 
The function $f$ satisfies the following KL inequality at every $\bar x\in \crit(f)$: There are $\eta \in (0,\infty]$ and a neighborhood $U(\bar x)$ of $\bar x$ such that
		\begin{equation}
			\label{eq:kl-ineq}	     	
			\|\nabla{f}(x)\|\geq \sC |f(x)-f(\bar x)|^\theta \quad \forall~x \in U(\bar x) \cap \{x \in \Rn: 0 < |f(x) - f(\bar x)| < \eta\},
		\end{equation}
		
		for some $\sC>0$ and $\theta\in[\half,1)$.		
\end{assumption}
This assumption applies to a wide range of functions commonly encountered in practical optimization scenarios. Specifically, proper closed semi-algebraic functions satisfy the KL inequality \eqref{eq:kl-ineq} at every point \cite{AttBolRedSou10}. Examples of functions and problems satisfying this assumption comprise least-squares problems \cite{bolte2017error}, logistic regression \cite{li2018calculus}, quadratic optimization problems \cite{liu2019quadratic}, and polynomial functions \cite{DAcuntoKurdyka2005}. These cases highlight the broad applicability of the KL inequality in various areas of optimization.

\subsection{The Proposed Framework and Convergence Results}
\label{subsec:framework+conv}
We now formally present our KL-based analysis framework.
\begin{mdframed}[
  linecolor=black!100,%red!50!black,      % Border color
  linewidth=.2mm,          % Border width
  backgroundcolor=bluep!1,% Background color
  roundcorner=0pt,       % Corner radius
  innerleftmargin=10pt,   % Left margin
  innerrightmargin=15pt,  % Right margin
  innertopmargin=10pt,    % Top margin
  innerbottommargin=10pt  % Bottom margin
]
\underline{\textbf{A KL-based Analysis Framework}}. Let $K \in \N$, $\{x^k\}_k \subset \Rn$, $\{\beta_k\}_k \subset \R$ and $\{I_k\}_k \subseteq \N$ be given.
\vspace{2mm}
\begin{enumerate}[label=\textup{\textrm{(A.\arabic*)}},topsep=0pt, leftmargin = 28pt,itemsep=0ex,partopsep=0ex]
	\item \label{C1} \textbf{Approximate descent.} There exist $a_1>0$ and $\{p_k\}\subseteq \R_+$ such that for any integer $k \geq K$,
	\[
		f({x}^{I_{k+1}}) +a_1 \beta_k \|\nabla{f}({x}^{I_k})\|^2 \leq f({x}^{I_k})  +  p_k.
\]
	\item \label{C2} \textbf{Gradient-bounded update.}  There exist $a_2>0$ and $\{q_k\}\subseteq \R_+$ such that for any integer $k \geq K$,
\[
		\max_{I_k < i \leq I_{k+1}}\|x^{i}-x^{I_k}\| \leq a_2 \beta_k\|\nabla f(x^{I_k})\| + q_k.
\]
	
	\item \label{C3} \textbf{Parameter requirements (I).} For some constant  $\bar \beta>0$,
	\[
	0<\beta_k \leq \bar \beta,\quad \sum_{k=1}^{\infty}\, \beta_k = \infty,\quad \sum_{k=1}^{\infty}\,p_k <\infty,\quad \text{and}\quad \sum_{k=1}^{\infty}\, \beta_k^{-1}q_k^2 <\infty .
	\]
	
	\item \label{C4} \textbf{Parameter requirements (II).} For some constants $\mu\in (0,1)$ and $\bar \beta>0$,
	\[
		0<\beta_k \leq \bar \beta,\quad \sum_{k=1}^{\infty}\, \beta_k = \infty,\quad \sum_{k=1}^{\infty}\,\beta_k \left[{\sum}_{t=k}^{\infty}\,p_t\right]^\mu<\infty,\quad \text{and}\quad \sum_{k=1}^{\infty}\,\beta_k \left[{\sum}_{t=k}^{\infty} \,\beta_t^{-1}q_t^2\right]^\mu <\infty .
\]
\end{enumerate}
\end{mdframed}
We commented already on the algorithmic conditions \ref{C1}--\ref{C2} in the introduction. The parameter conditions \ref{C3}--\ref{C4} impose constraints on the algorithm's parameters, typically relating to the choice of step sizes. In stochastic and distributed settings, the sequences $\{p_k\}_k$ and $\{q_k\}_k$ often emerge as higher-order terms of the step sizes (see, \cref{sec:finite sum}). \ref{C3}--\ref{C4} then effectively translate into step size requirements, as demonstrated in the simplified framework in \cref{subsec:special case}. 

We now present the main convergence theorem, which provides pivotal  understanding of the asymptotic behavior of $\{x^k\}_k$. The proof can be found in \cref{proof:main}.

\begin{thm}
    \label{thm:main}
    Let \cref{as:func-1} hold and suppose that the iterates $\{x^k\}_k$ satisfy \ref{C1}--\ref{C3}.
    Then, the following statements hold:
    \begin{enumerate}[label=\textup{(\alph*)},topsep=0pt,itemsep=0ex,partopsep=0ex]
    \item\label{weak-conv} It holds that $\lim_{k\to\infty}\|\nabla f(x^k)\|=0$ and the sequence $\{f(x^k)\}_k$ converges to some $f^*\in\R$.     
    \item\label{strong-conv} If, in addition, \cref{as:kl} holds and \ref{C4} is satisfied, then $\{x^k\}_k$ either converges to some stationary point of $f$ or $\lim_{k \to \infty} \|x^k\| = \infty$.
    \end{enumerate}
\end{thm}
\begin{remark}
Let $\acc(\{x^k\}_k)$ denote the set of accumulation points of $\{x^k\}_k$, i.e.,
    \begin{equation}\label{accumulation-point}
\acc(\{x^k\}_k):=\{x\in\Rn : {\liminf}_{k\to\infty}\|x^k-x\|=0\}.
\end{equation}
    \begin{itemize}[leftmargin = 20pt]
    \item If $\acc(\{x^k\}_k)$ is non-empty, then \cref{thm:main} {(a)} implies that $\acc(\{x^k\}_k) \subseteq \crit(f)$, i.e., every accumulation point of $\{x^k\}_k$ is a stationary point of $f$. Furthermore, this guarantees $\liminf_{k\to\infty} \|x^k\|<\infty$, ruling out the situation $\lim_{k \to \infty} \|x^k\| = \infty$. Combining this observation with \cref{thm:main} {(b)}, we conclude that $\{x^k\}_k$ converges to a stationary point $x^*\in\crit(f)$.
    \item If the function $f$ is coercive, \cref{thm:main} (b) readily implies that  $\{x^k\}_k$ converges to some $x^*\in\crit(f)$. Let us explain how to show this. Based on \ref{C1} and $\sum_{k=1}^\infty p_k <\infty$, we have 
        \[
        f(x^{I_{k+1}}) \leq f(x^{I_{k}}) + p_k \quad \Longrightarrow \quad f(x^{I_{k}}) \leq f(x^{I_{1}}) + {\sum}_{i=1}^\infty \; p_i \quad \text{for all $k\geq 1$}. 
        \]
        Hence, by coercivity, we conclude that $\{x^{I_k}\}_k$ is bounded. This guarantees that $\{x^{I_k}\}_k$ has at least one convergent subsequence, and thus, $\acc(\{x^k\}_k)$ is non-empty. Based on former discussion and \cref{thm:main} (b), $\{x^k\}_k$ converges to some $x^*\in\crit(f)$.
    \end{itemize}
\end{remark}
\subsection{A Special Case and Convergence Rates}\label{subsec:special case}
This subsection presents a unified framework for deriving convergence rates of iterative methods employing polynomial step sizes. Specifically, we consider a special case of the conditions \ref{C1}--\ref{C3} by setting
\[K=1, \quad I_k = k,\quad \beta_k = \alpha_k,\quad p_k=b_2 \alpha_k^p,\quad q_k=c_2 \alpha_k^q, \quad \text{where} \quad p,q>1 \; \text{and} \; \alpha_k=\alpha/k^\gamma,\,\alpha>0,\,\gamma\in(0,1]. \]
Then, our proposed framework simplifies to 
\begin{mdframed}[
  linecolor=black!100,%red!50!black,      % Border color
  linewidth=.2mm,          % Border width
  backgroundcolor=bluep!1,% Background color
  roundcorner=0pt,       % Corner radius
  innerleftmargin=10pt,   % Left margin
  innerrightmargin=15pt,  % Right margin
  innertopmargin=10pt,    % Top margin
  innerbottommargin=10pt  % Bottom margin
]
\begin{enumerate}[label=\textup{\textrm{(R.\arabic*)}},topsep=0pt, leftmargin = 28pt,itemsep=0ex,partopsep=0ex]
	\item \label{R1} \textbf{Approximate descent.} There exist $b,\tilde b>0$ and $p > 1$ such that for all $k \geq 1$, 
	\[
		f({x}^{k+1}) +b \cdot \alpha_k \|\nabla{f}({x}^{k})\|^2 \leq f({x}^{k})  + \tilde b \cdot  \alpha_k^p .
\]
	
	\item \label{R2} \textbf{Gradient-bounded update.} There exist $c>0$ and $q>1$ such that for all $k \geq 1$, 
\[
	\|x^{k}-x^{k+1}\| \leq c \cdot \alpha_k\|\nabla f(x^{k})\| + c \cdot \alpha_k^q.
\]
	\item \label{R3} \textbf{Parameter requirements.} $\alpha_k = \alpha/k^\gamma$ with $\alpha>0$, $\gamma \in (\tfrac{1}{p},1]$ and $p+1 \leq 2q$.
\end{enumerate}
\end{mdframed}
For non-descent optimization methods, such as random reshuffling and distributed  methods (cf., \cref{sec:finite sum}), it is frequently observed that the parameters $p$ and $q$ satisfy $p+1 \leq 2q$, as stated in condition \ref{R3}. To enhance clarity, we derive convergence rates under this specific condition. Note that our results can be easily extended to more general cases. Prior to deriving the convergence rates, we introduce the uniformized KL property.
\begin{lem}[Uniformized KL property]\label{lem:uniform-kl}
    Suppose \cref{as:func-1} and \ref{as:kl} hold and let $\acc(\{x^k\}_k)$ be bounded. Then, we have $f(\bar x) = \lim_{k\to\infty} f(x^k)= f^*$ for all $\bar x \in \acc(\{x^k\}_k)$. In addition, there are $\varepsilon >0$ and $\eta >0$ such that for all $x\in V_{\varepsilon,\eta}:=\{x\in\Rn:\dist(x,\acc(\{x^k\}_k))<\varepsilon\} \cap \{x\in\Rn:0<|f(x) - f^*|<\eta\}$, it holds that
    \[
        \|\nabla f(x)\| \geq \sC|f(x) - f^*|^\theta \quad \text{for some $\sC>0$ and $\theta\in[\tfrac12,1)$}.
    \]
\end{lem}
\begin{proof}
    Notice that \ref{C1}--\ref{C3} hold with $I_k = k$, $a_1 = b$, $a_2 = c$, $\beta_k = \alpha_k$, $p_k = \tilde b \alpha_k^p$, and $q_k = c \alpha_k^q$. Then, by \cref{thm:main} (a), we conclude that every accumulation point is a stationary point and $\{f(x^k)\}_k$ converges to some $f^*\in\R$. Hence, it holds that $\acc(\{x^k\}_k) \subseteq \crit(f)$ and $f(\bar x) = f^*$ for all $\bar x\in\acc(\{x^k\}_k)$ by the continuity of $f$. According to \cite[Lemma 6]{BolSabTeb14} and \cref{as:kl}, we establish the uniformized KL property \cref{lem:uniform-kl}.
\end{proof}
\Cref{lem:uniform-kl} plays a crucial role in deriving convergence rates when the sequence of iterates $\{x^k\}_k$ does not necessarily converge to a single stationary point. In such cases, as outlined in \cref{thm:convergence rate} (b), we assume that the set of accumulation points $\{x^k\}_k$ is bounded. This bounded iterates assumption can be relaxed if we impose stricter conditions on the step size selection (i.e., $\gamma > \frac{2}{p+1}$).

We now present our main results concerning the convergence rates of $\{x^k\}_k$ under the framework \ref{R1}--\ref{R3}. The proof of these results is provided in \cref{app:proof-rates}. 
\begin{thm}[Convergence rates]\label{thm:convergence rate} Let \cref{as:func-1} and \cref{as:kl} hold and suppose that the sequence $\{x^k\}_k$ satisfies the conditions \ref{R1}--\ref{R3}. The following statements are valid.
\begin{enumerate}[label=\textup{(\alph*)},topsep=1ex,itemsep=0ex,partopsep=0ex, leftmargin = 25pt]
\item If $\acc(\{x^k\}_k)$ is non-empty and $\gamma > \frac{2}{p+1}$, then $x^k \to x^*\in\crit(f)$ and $f(x^k) \to f(x^*) = f^*$ and we have
\[ |f(x^k)-f^*| = \cO(k^{-\psi(\theta,\gamma)}), \quad \|\nabla f(x^k)\|^2 =  \cO(k^{-\psi(\theta,\gamma)}), \quad \|x^k - x^*\| = \cO(k^{-\varphi(\theta,\gamma)}), \quad \gamma \in (\tfrac{2}{p+1},1),  \]
where $\theta$ denotes the KL-exponent of $f$ at $x^*$ and
\[ \psi(\theta,\gamma) = \min\{p\gamma-1, \tfrac{1-\gamma}{2\theta-1}\}, \quad \varphi(\theta,\gamma)=\min\{\tfrac{(p+1)\gamma}{2}-1,\tfrac{(1-\gamma)(1-\theta)}{2\theta-1}\}. \]
Moreover, if $\theta = \frac12$, $\gamma = 1$, $\alpha > \frac{2(p-1)}{b\sC^2}$, it holds that $|f(x^k)-f^*| = \mathcal O(k^{-(p-1)})$, $\|\nabla f(x^k)\|^2= \mathcal O(k^{-(p-1)})$, and $\|x^k - x^*\| = \cO(k^{-\frac{p-1}{2}})$.
\item If $\{x^k\}_k$ is bounded, then the limit $f(x^k) \to f^*$ and the rates $|f(x^k)-f^*| = \cO(k^{-\psi(\theta,\gamma)})$ and $\|\nabla f(x^k)\|^2 =  \cO(k^{-\psi(\theta,\gamma)})$ continue to hold in the case $\gamma \in (\frac1p,\frac{2}{p+1}]$ with $\theta$ being the uniformized KL-exponent over the set $\mathcal A(\{x^k\}_k)$ (cf. \cref{lem:uniform-kl}). 
\end{enumerate}
\end{thm}
\begin{remark}[Optimal choice of $\gamma$]\label{rem:convergence rate}
    The results in \cref{thm:convergence rate} provide guidance for selecting the optimal step size parameter $\gamma$, when information about the local geometry of the function is available  (i.e., when the KL-exponent is known). 
    In particular, if the exact KL-exponent $\theta\in[\frac12,1)$ is known, then, based on \cref{thm:convergence rate}, the optimal step size parameter $\gamma^*$ to maximize the convergence rate is given by
        \[
        \gamma^* = \frac{2\theta}{(2\theta-1)p+1} \quad \implies \quad \psi(\theta,\gamma^*) = \frac{p-1}{(2\theta-1)p+1} \quad \text{and} \quad \varphi(\theta,\gamma^*) = \frac{(p-1)(1-\theta)}{(2\theta-1)p+1}.
        \]
        Noting $\gamma^* > \frac{2}{p+1}$, this choice of $\gamma$ guarantees convergence of $\{x^k\}_k$ to some $x^*\in\crit(f)$. 
\end{remark}

\section{Application Area I : Stochastic Approximation Methods}\label{sec:sgd}
In areas of stochastic approximation and online learning, the objective function $f$ corresponds to a data-driven predictive learning task in the form of
\be \label{eq:prob-exp} f(x):=\Exp_{\xi\sim\Xi}[F(x,\xi)]=\int_{\Xi} F(x,\xi)\,\mathrm{d}\mu(\xi)% \quad \text{or} \quad f(x):=\frac{1}{N}{\sum}_{i=1}^{N}f_i(x), 
\ee
where the underlying probability space $(\Xi,{\cal H},\mu)$ is usually unknown.

In this section, we establish convergence results for $\SGD$ applying the framework presented in \cref{subsec:new_framework}. $\SGD$ generates a stochastic process $\{\sx^k\}_k$ via the update rule
\begin{equation}
	\label{algo:sgd}
	\sx^{k+1} = \sx^k - \alpha_k \sg^k, \quad \forall\, k \geq 1,
\end{equation}
where $\sx^1 \equiv x^1$ is a deterministic initial point. Throughout this section, we will use bold letters to describe random variables ${\sx} : \Omega \to \Rn$, while lowercase letters are typically reserved for realizations of a random variable, $x = {\sx}(\omega)$, or deterministic parameters. In addition, we assume that there is a sufficiently rich probability space $(\Omega,\mathcal F,\Prob)$ that allows modeling and studying the iterates generated by $\SGD$ in a unified way. Hence, each stochastic approximation of $\nabla f(x^k)$ is understood as a realization of a random vector  $\sg^k: \Omega \to \Rn$. 

We further consider the natural filtration ${\cal F}_k:=\sigma(\sx^1,\sx^2,\ldots,\sx^k)$; each iterate $\sx^k$ is then $\mathcal F_{k}$-measurable for $k\geq 1$. Next, let us define the stochastic error term $\se^k := \nabla f(\sx^k) - \sg^k $ and let us introduce our main assumptions on $\{\se^k\}_k$.
\begin{assumption} \label{as:sgd}
	Given the probability space $(\Omega,{\cal F},\{{\cal F}_k\}_k,\Prob)$, the stochastic errors $\{\se^k\}_k$ are assumed to satisfy:
\begin{enumerate}[label=\textup{(\alph*)},topsep=0ex,itemsep=0ex,partopsep=0ex, leftmargin = 25pt]
 \item \textbf{Mean zero:} $\Exp[\se^k \mid \mathcal F_{k}]=0$ almost surely for all $k \geq 1$.
 \item \textbf{Bounded variance:} $\Exp[\|\se^k\|^2 \mid \mathcal F_{k}] \leq \sigma^2$ almost surely for all $k \geq 1$.
\end{enumerate}
\end{assumption}
This assumption is fairly standard in the analysis of stochastic methods \cite{borkar2009stochastic,tadic2015convergence,milzarek2023convergence}. 
We note that the convergence results presented in this subsection can be covered by our recent preprint \cite{qiu2024convergence}, which proves the global convergence, iterate convergence, and convergence rates of $\SGD$ with momentum. However, the analysis therein requires some nontrivial adaptation to the framework, including the use of auxiliary iterates and a merit function. For simplicity, we only present the convergence results tailored to $\SGD$, which does not require additional adaptations to the framework.

Setting $\sT := \frac{1}{50\sL}$, we define the sequence of indices $\{I_{k}\}_k$ inductively via
 \[
        I_1 = 1 \quad \text{and} \quad
        I_{k+1} = \max\Big\{I_k+1,\; \sup\big\{n \geq k: {\sum}_{i=k}^{n-1} \alpha_i \leq \sT\big\} \Big\},
        \quad k \geq 1.
 \]
We also consider the aggregated stochastic error terms $\{\scs_k\}_k$ and the associated event $\mathcal S_\mu$: 
 \[
    \scs_k := \max_{I_k < t \leq I_{k+1}} \Big\| {\sum}_{i=I_k}^{t-1} \alpha_i \se^i \Big\| 
    \quad \text{and} \quad
    \cS_\mu = \Big\{\omega \in \Omega: {\sum}_{k=1}^\infty \Big({\sum}_{i=1}^k \scs_i^2(\omega)\Big)^\mu < \infty \Big\}.
 \]
 Before presenting the main convergence results, we first provide the approximate descent and gradient-bounded update properties for $\{x^{I_k}\}_k$. For their proofs, we simply refer to the relevant parts in \cite{qiu2024convergence}.
 \begin{proposition}\label{lem:sgd-bound1}
     Suppose \Cref{as:func-1} and \ref{as:sgd} hold and let $\{\sx^k\}_k$ be generated by $\SGD$ with step sizes $\{\alpha_k\}_k$ satisfying 
	 \begin{equation} \label{eq:sgd-stepsizes1}
    \alpha_k > 0, \quad {\sum}_{k=1}^\infty \alpha_k = \infty \quad \text{and} \quad 
    {\sum}_{k=1}^\infty \alpha_k^2 \Big({\sum}_{i=1}^k \alpha_i\Big)^{r} < \infty, \quad \text{for some} \; r >1.
 \end{equation}
	Then, we have $\Prob(\cS_\mu) = 1$ for all $\mu>1/r$. In addition, for any $\omega \in \cS_\mu$, we consider the realizations $x^k \equiv \sx^k(\omega)$ and $s_k \equiv \scs_k(\omega)$. There then exists $K = \sK(\omega)$ such that for all $k \geq K$, it holds that:
\begin{enumerate}[label=\textup{(\alph*)},topsep=0ex,itemsep=0ex,partopsep=0ex, leftmargin = 25pt]
\item  $f(x^{I_{k+1}}) + \frac{\sT}{100} \|\nabla f(x^{I_k})\|^2 \leq f(x^{I_k}) + \frac{6}{\sT} s_k^2$.
\item $\max_{I_k < i \leq I_{k+1}} \|x^i - x^{I_k}\| \leq \frac{20}{19} (\sT \|\nabla f(x^{I_k})\| + s_k)$.
\end{enumerate}
\end{proposition}
\begin{proof}
The proof of \cref{lem:sgd-bound1} mainly relies on the work \cite{qiu2024convergence} where the stochastic gradient method with momentum is considered. First, $\Prob(\cS_\mu) = 1$ follows from \cite[Eq. (4.10)]{qiu2024convergence}. The bound in (a) restates \cite[Lemma 3.4]{qiu2024convergence} and (b) follows from \cite[Eq. (A.5)]{qiu2024convergence}.
\end{proof}
\begin{remark}
    The step size condition ${\sum}_{k=1}^\infty \alpha_k^2 ({\sum}_{i=1}^k \alpha_i)^{r} < \infty$ for some $r >1$ is not very common; it is satisfied by polynomial step sizes $\alpha_k \sim k^{-\gamma}$ with $\gamma\in (\tfrac23,1]$. Indeed, using $\alpha_k = k^{-\gamma}$, $\gamma\in(\tfrac23,1)$, and the integral test, we have
\[
{\sum}_{k=1}^\infty \alpha_k^2 \Big({\sum}_{i=1}^k \alpha_i \Big)^{r} = {\sum}_{k=1}^\infty k^{-2\gamma} \Big({\sum}_{i=1}^k i^{-\gamma} \Big)^{r} \leq {\sum}_{k=1}^\infty k^{-2\gamma} \left(1 + \tfrac{k^{1-\gamma}}{1-\gamma}\right)^{r} = \cO\Big({\sum}_{k=1}^\infty k^{-2\gamma+r(1-\gamma)}\Big).
\]
The right hand side is summable when $1<r<\frac{2\gamma-1}{1-\gamma}$. The case $\gamma=1$ is trivial.
\end{remark}
Based on the key properties presented above, the global and iterate convergence of $\SGD$ follows quite easily from \cref{thm:main}. 
\begin{thm} \label{thm:sgd-main}
 Suppose \Cref{as:func-1} and \ref{as:sgd} hold and let $\{\sx^k\}_k$ be generated by $\SGD$ with step sizes $\{\alpha_k\}_k$ satisfying \eqref{eq:sgd-stepsizes1}. Then, the following statements hold.
\begin{enumerate}[label=\textup{(\alph*)},topsep=0ex,itemsep=0ex,partopsep=0ex, leftmargin = 25pt]
 \item $\|\nabla f(\sx^k)\| \to 0$ a.s. and  $\{f(\sx^k)\}_k$ converges to some $\sct{f}^*: \Omega \to \mathbb{R}$ a.s.\;\!.
 \item If, in addition, \cref{as:kl} holds, then the event
 \[
  \left\{\omega\in\Omega: {\lim}_{k\to\infty}\,\|\sx^k(\omega)\|=\infty \;\, \text{or}\;\, \sx^k(\omega) \to \sx^*(\omega)\in\crit(f) \right\} \quad   \text{occurs a.s.\;\!.}
 \]
 \end{enumerate}
\end{thm}
\begin{proof}
     Let us fix $\mu\in(1/r,1)$, pick an arbitrary $\omega \in \cS_\mu$ and consider the realizations $x^k \equiv \sx^k(\omega)$ and $s_k \equiv \scs_k(\omega)$. Thanks to \cref{lem:sgd-bound1}, the algorithmic conditions \ref{C1}--\ref{C2} are satisfied with 
     \[
     a_1 = \sT/ 100,\quad a_2 = 20\sT/19,\quad \beta_k = 1,\quad p_k = (6/\sT) s_k \quad \text{and} \quad q_k = (20/19)s_k.
     \]
     Let us now check \ref{C3}. By the definition of $\cS_\mu$ (and due to $\omega \in \cS_\mu$), we have ${\sum}_{k=1}^\infty ({\sum}_{i=1}^k s_i^2)^\mu < \infty$. This implies that $\sum_{k=1}^\infty s_k^2 <\infty$, which verifies \ref{C3}. Applying \cref{thm:main} (a), we complete the proof of the statement (a).  
     
     The relation ${\sum}_{k=1}^\infty ({\sum}_{i=1}^k s_i^2)^\mu < \infty$ also verifies \ref{C4}. Under \cref{as:kl}, statement (b) simply follows from the convergence result in \Cref{thm:main} (b).  
\end{proof}
\begin{remark}
To interpret \cref{thm:sgd-main} (b), let us define an event $\mathcal{X}$ that represents the non-diverging trajectories of $\{\sx^k\}_k$:
\[\mathcal{X} = \{\omega \in \Omega: \acc(\{\sx^k(\omega)\}_k) \neq \emptyset\} = \{\omega \in \Omega: {\liminf}_{k\to\infty} \|\sx^k(\omega)\| < \infty\}.\] 
Then, \cref{thm:sgd-main} (b) asserts that $\{\sx^k\}_k$ converges a.s. on the event $\mathcal{X}$. Thus, $\{\sx^k\}_k$ converges a.s. to some stationary point under the extra assumption $\Prob(\mathcal{X})=1$. This is significantly weaker than the typical a priori boundedness assumption used in other KL-based convergence analyses of stochastic methods \cite{borkar2009stochastic,milzarek2023convergence,tadic2015convergence}.    
\end{remark}

\section{Application Area II : Finite-sum Optimization}\label{sec:finite sum}
This section demonstrates the versatility of our framework by establishing iterate convergence of several widely-used algorithms for finite-sum optimization problems. Let us consider the optimization problem
\begin{equation}
	\label{LO} \min_{x\in \Rn}~f(x) := \frac{1}{n}\sum_{i=1}^{n}f_i(x),%\tag{LO} 
\end{equation}
where $f_i: \Rn \to \R$ is differentiable but not necessarily convex for all $i\in [n] := \{1,\dots,n\}$. We make the following assumption on the objective function.
\begin{assumption}%[Objective function]
	\label{as:func}
	Every component function $f_i:\Rn \to \R$, $i\in[n]$, is ${\sf L}$-smooth and bounded from below by $\bar f$.
\end{assumption}
This assumption can be guaranteed if every component function $f_i$ is $\sL_i$-smooth and bounded from below by $\bar f_i \in \R$. In this case, we set $\sL=\max_{i\in[n]} \sL_i$ and $\bar f:=\min_{i\in[n]}\bar f_i$.

\subsection{Decentralized Gradient Descent}
\label{sec:dgd}
While prior research has examined the convergence properties of decentralized gradient descent ($\DGD$), primarily in convex settings \cite{nedic2009distributed,yuan2016convergence,zeng2018nonconvex}, a comprehensive understanding of its behavior in nonconvex scenarios remains elusive.
Zeng and Yin \cite{zeng2018nonconvex} and Daneshmand et al. \cite{daneshmand2020second} have analyzed $\DGD$ with constant step sizes under the KL property. However, as highlighted in \cite[Proposition 2]{zeng2018nonconvex}, their results do not guarantee the consensus of limit points generated by the agents' iterates. Furthermore, \cite[Theorem 3.9]{daneshmand2020second} establishes convergence to a neighborhood of the set of critical points $\crit(f)$, rather than convergence to some stationary point.

In this subsection, we address these limitations by employing our proposed KL-based analysis framework. This approach enables us to establish the convergence of $\DGD$ iterates to a stationary point consensually with convergence rate guarantees.
 
Suppose that there are $n$ agents and the goal is to solve problem \eqref{LO} collaboratively. Let $x_i^1=x^1$, $i\in[n]$ be the initial points. The update step of $\DGD$ \cite{nedic2009distributed} is given as
\begin{equation}
	\label{algo:dgd}	
	x^{k+1}_i = \sum_{j=1}^n w_{ij}x^k_j - \alpha_k \nabla f_i(x_i^k),\quad \forall\, i\in[n] \; \text{and}\;  k \geq 1 % = (\text{Id} - \alpha_k {\cal G})(x^k)
\end{equation}
where $w_{ij} \geq 0$ is the weight that neighbor $j$ assigns to the agent $i$ and $\{\alpha_k\}_k$ are step sizes satisfying
\begin{equation}
\label{eq:dgd-stepsize-req1}
\alpha_k\in\Big(0,{\frac{1}{{\sf L}n}}\Big), \quad \alpha_{k + 1} \leq \alpha_{k},\quad {\sum}_{k=1}^\infty \alpha_k = \infty, \quad {\sum}_{k=1}^\infty \alpha_k^3 <\infty, \quad \text{and} \quad \lim_{k \to \infty} \frac{\alpha_{k+1}}{\alpha_k} = 1.
\end{equation}
We make the following assumption on the weights $\{w_{ij}\}_{i,j}$, $i,j\in[n]$.

\begin{assumption}%[Graph structure]
	\label{as:matrix}
	The matrix $W=[w_{ij}]\in\R^{n\times n}$ is symmetric and has the following properties.
\begin{enumerate}[label=\textup{(\alph*)},topsep=0ex,itemsep=0ex,partopsep=0ex, leftmargin = 25pt]
		\item $W$ is doubly stochastic, i.e., $\nul(I-W)=\spa(\one)$.
	%	\item $W$ is compatible with graph $\mathcal{G}$, i.e., $w_{ij}>0\iff(i,j)\in\mathcal{E}$
		\item $W$ has $1$-bounded eigenvalues, i.e., $-I\prec W \preceq I.$
	\end{enumerate}
\end{assumption} 
Here, $\one \in \Rn$ is the vector of all ones. Condition (a) can be equivalently written as $W\one=\one$ and $\one^\top W = \one^\top$. Condition (b) means that the second largest magnitude eigenvalue of $W$ lies in the interval $(0,1)$. This assumption is standard in the study of distributed optimization, see, e.g., \cite{shi2015extra,nedic2009distributed,yuan2016convergence,huang2023distributed}. To simplify the notations, we denote the average iterates of the agents and the corresponding accumulation points set as 
\[
\bar x^k:=\frac1n\,{\sum_{i=1}^n x_i^k}\quad \text{and} \quad \acc(\{\bar x^k\}_k):=\{x\in\Rn : {\liminf}_{k\to\infty}\|\bar x^k-x\|=0\}.\]
Before presenting the main convergence results, we first provide an upper bound for the consensus error and the descent-type property for $\{\bar x^k\}_k$. The proofs are postponed to \cref{proof:lemma:dgd} and \cref{proof:dgd-descent}, respectively.
\begin{proposition}
\label{lemma:dgd}
Let \cref{as:func} and \ref{as:matrix} hold and let $\{x_i^k\}_{k\geq 1,i\in[n]}$ be generated by $\DGD$ with step sizes $\{\alpha_k\}_k$ satisfying \eqref{eq:dgd-stepsize-req1}. Then there exists $\sG>0$ such that %$\|\bar x^k \one^T -  \bx^k\|^2_F \leq \sG \alpha_k^2$. 
$\sum_{i=1}^n \|x_i^k- \bar x^k\|^2 \leq \sG \alpha_k^2$. 
\end{proposition}
\begin{proposition}
		\label{prop:dgd}
	Let \cref{as:func} and \ref{as:matrix} hold and let $\{x_i^k\}_{k\geq 1,i\in[n]}$ be generated by $\DGD$. Further assume that the step sizes $\{\alpha_k\}_k$ satisfy \eqref{eq:dgd-stepsize-req1}. Then, there is $\sG>0$ such that for all $k\geq 1$, it holds that
 \begin{enumerate}[label=\textup{(\alph*)},topsep=0ex,itemsep=0ex,partopsep=0ex, leftmargin = 25pt]
\item $f(\bar x^{k+1})   \leq f(\bar x^k) - \frac{\alpha_k}{2}\|\nabla f(\bar x^k)\|^2 + \frac{\sL\sG}{2n}\alpha_k^3$.	
\item $\|\bar x^k-\bar x^{k+1}\| \leq \alpha_k \|\nabla f(\bar x^k)\| + \frac{\sL\sqrt{\sG}}{\sqrt{n}}\alpha_k^2$.
\end{enumerate}
\end{proposition}
Based on the properties presented above, we are able to apply \cref{thm:main} to establish the convergence for the averaged iterates $\{\bar x^k\}_k$. Then, by using \cref{lemma:dgd}, we can show that each agent converges consensually to the stationary point of $f$. Importantly, this convergence result holds without requiring a unique or isolated stationary point.
\begin{thm}[Convergence of $\DGD$]
\label{thm:dgd}
Let \cref{as:func} and \ref{as:matrix} hold and let $\{x_i^k\}_{k\geq 1,i\in[n]}$ be generated by $\DGD$ with the step sizes $\{\alpha_k\}_k$ satisfying 
\eqref{eq:dgd-stepsize-req1} and
\begin{equation}
	\label{eq:step-size-dgd}
{\sum}_{k=1}^{\infty}\alpha_k\Big({\sum}_{i=k}^{\infty} \alpha_i^3\Big)^\mu<\infty\quad \text{for some }\mu\in(0,1).
\end{equation}
Then, $\lim_{k\to\infty}\|\nabla f(x^k_i)\|=0$ and $\{f(x_i^k)\}_k$ converges to some $f^*\in\R$ for all $i\in[n]$. Moreover, if \cref{as:kl} holds, then $\{x_i^k\}_{k}$ either converges to some stationary point of $f$ consensually or we have $\|x_i^k\| \to \infty$ for all $i\in[n]$.  
\end{thm}

\begin{proof}
Clearly, all conditions in \cref{lemma:dgd} and \ref{prop:dgd} are satisfied.
Using \cref{prop:dgd}, we verify \ref{C1}--\ref{C2} with 
\begin{equation*}
	I_k = k, \quad \beta_k = \alpha_k, \quad p_k = {\textstyle\frac{\sL\sG}{2n}} \cdot \alpha_k^3, \quad q_k= {\textstyle\frac{\sL\sqrt{\sG}}{\sqrt{n}}} \cdot \alpha_k^2, \quad a_1 = {\textstyle \frac12}, \quad \text{and}\quad a_2 =1.
\end{equation*}
 Moreover, the step sizes $\{\alpha_k\}_k$ satisfy \ref{C3}--\ref{C4}. It follows from \cref{thm:main} (a) that $\|\nabla f(\bar{x}^k)\|\to 0$ and $f(\bar x^k)\to f^*\in\R$ as $k$ tends to infinity. Utilizing the Lipschitz continuity of $f$, we have for all $i\in[n]$ that $\|\nabla f(x^k_i)\| \leq \|\nabla f(\bar x^k)\| + \sL\|x^k_i- \bar x^k\|$, together with $\|\bar{x}^k-x_i^k\| \to 0$ (implied by \cref{lemma:dgd}), this yields $\|\nabla f(x^k_i)\| \to 0$ as $k$ tends to infinity. Based on the Lipschitz continuity of $\nabla f$ and \cite[Theorem 2.1.5]{nesterov2018lectures}, it further follows
 \begin{equation}
     \label{eq:dgd-Lipschitz}
     \begin{aligned}
          |f(\bar x^k) - f(x_i^k)| &\leq \max\{\|\nabla f(\bar x^k)\|,\|\nabla f(x_i^k)\|\}\cdot \|\bar x^k - x_i^k\| + \tfrac{\sL}{2}\|\bar x^k - x_i^k\|^2 \\&\leq \tfrac{1}{2\sL}\max\{\|\nabla f(\bar x^k)\|^2,\|\nabla f(x_i^k)\|^2\} + \sL\|\bar x^k - x_i^k\|^2,
     \end{aligned}
 \end{equation}
 where the last line is due to $\|a\|\cdot \|b\| \leq \frac{1}{2\sL}\|a\|^2 + \frac{\sL}{2} \|b\|^2$. 
 Then, based on $\|\nabla f(\bar x^k)\| \to 0$, $\|\nabla f(x_i^k)\|\to 0$ and $\|\bar x^k - x_i^k\|\to 0$, we have $|f(\bar x^k) - f(x_i^k)|\to 0$. This, along with $f(\bar x^k)\to f^*$, implies that $f(x_i^k) \to f^*$ for all $i\in[n]$.

 If there is an agent $i\in[n]$ such that $\|x_i^k\|\to\infty$, then from $\|\bar x^k - x_i^k\|\to 0$, we have $\|\bar x^k\|\to \infty$ and $\|x_j^k\|\to\infty$ for all $j\in[n]$. Now, let us consider the case where $\liminf_{k\to\infty} \|\bar x^k\|<\infty$ and \cref{as:kl} holds. \cref{thm:main} (b) implies that $\{\bar x^k\}_k$ converges to some stationary point $x^*$ of $f$. Since $\|\bar x^k - x_i^k\|\to 0$, we conclude that $\{x_i^k\}_k$ converges to $x^*$ for all $i\in[n]$ as well.
\end{proof}

To the best of our knowledge, \cref{thm:dgd} establishes the first iterate convergence of $\DGD$ in the nonconvex setting. This is achieved by utilizing the proposed KL-based analysis framework. Furthermore, our framework enables us to characterize the convergence rates for $\DGD$.

\begin{thm}\label{thm:rate-dgd}
	Let \cref{as:kl}, \ref{as:func} and \ref{as:matrix} hold and let $\{x_i^k\}_{k\geq 1,i\in[n]}$ be generated by $\DGD$ with step sizes $\{\alpha_k\}_k$ of the form $\alpha_k=\alpha/k^\gamma, \alpha>0, \gamma\in(0,1]$. %\[\alpha_k=\alpha/k^\gamma, \quad \alpha>0 \quad \text{and} \quad \gamma\in(0,1].\] 
	Then, the following statements hold for all $i\in[n]$.
 \begin{enumerate}[label=\textup{(\alph*)},topsep=0ex,itemsep=0ex,partopsep=0ex, leftmargin = 25pt]
  \item If $\acc(\{\bar x^k\}_k)$ is non-empty and $\gamma\in(\frac12,1)$, then $x_i^k \to x^*\in\crit(f)$ and $f(x_i^k)\to f^*:=f(x^*) $ and we have
	 \[
	  |f(x_i^k)-f^*| = \cO(k^{-\psi(\theta,\gamma)}), \quad \|\nabla f(x_i^k)\|^2 = \cO(k^{-\psi(\theta,\gamma)}),\quad   \|x_i^k - x^*\| = \cO(k^{-\varphi(\theta,\gamma)}),\quad  \gamma\in(\tfrac{1}{2},1),
   \]
   where $\theta$ denotes the KL-exponent of $f$ at $x^*$ and
   \[\psi(\theta,\gamma):=	\min\{3\gamma-1, {\textstyle \frac{1-\gamma}{2\theta-1} } \},\quad \varphi(\theta,\gamma):=	\min\{2\gamma-1, {\textstyle \frac{(1-\gamma)(1-\theta)}{2\theta-1}}\}.
	 \]
Moreover, if $\theta=\half$, $\gamma=1$ and $\alpha > 8/\sC^2$, then it holds that $|f(x_i^k)-f^*| = \cO(k^{-2})$, $\|\nabla f(x_i^k)\|^2 = \cO(k^{-2})$, and $\|x_i^k - x^*\| = \cO(k^{-1})$.
\item If $\{\bar x^k\}_k$ is bounded, then $f(x_i^k)\to f^*$ and $|f(x_i^k)-f^*| = \cO(k^{-\psi(\theta,\gamma)})$ and $\|\nabla f(x_i^k)\|^2 = \cO(k^{-\psi(\theta,\gamma)})$ continue to hold in the case $\gamma\in(\frac{1}{3},\frac{1}{2}]$ with $\theta$ being the uniformized KL-exponent over the set $\acc(\{\bar x^k\}_k)$.
\end{enumerate}
\end{thm}
\begin{proof}
    The recursions in \cref{prop:dgd} and specific step sizes $\alpha_k=\alpha/k^\gamma$ imply \ref{R1}--\ref{R3} with
    \[
    p=3,\quad q=2,\quad b=\tfrac12,\quad \tilde b=\tfrac{\sL\sG}{2n}\quad \text{and} \quad c=\max\{1,\tfrac{\sL\sqrt{\sG}}{\sqrt{n}}\}.
    \]
        When $\acc(\{\bar x^k\}_k)$ is non-empty and $\gamma\in(\frac12,1]$, \cref{thm:convergence rate} (a) yields $\bar x^k \to x^*\in\crit(f)$ and $\{f(\bar x^k)\}_k$ converges to $f^*:=f(x^*)$ with the rates 
     \[
	  |f(\bar x^k)-f^*| = \cO(k^{-\psi(\theta,\gamma)}), \quad \|\nabla f(\bar x^k)\|^2 = \cO(k^{-\psi(\theta,\gamma)}),\quad   \|\bar x^k - x^*\| = \cO(k^{-\varphi(\theta,\gamma)}),\quad  \gamma\in(\tfrac{1}{2},1),
   \]
   where $\theta$ denotes the KL-exponent of $f$ at $x^*$ and
   \[\psi(\theta,\gamma):=	\min\{3\gamma-1, {\textstyle \frac{1-\gamma}{2\theta-1} } \},\quad \varphi(\theta,\gamma):=	\min\{2\gamma-1, {\textstyle \frac{(1-\gamma)(1-\theta)}{2\theta-1}}\}.
	 \]
Moreover, if $\theta=\half$, $\gamma=1$ and $\alpha > 8/\sC^2$, then it holds that $|f(\bar x^k)-f^*| = \cO(k^{-2})$, $\|\nabla f(\bar x^k)\|^2 = \cO(k^{-2})$, and $\|\bar x^k - x^*\| = \cO(k^{-1})$. By \cref{lemma:dgd} and $\alpha_k = \alpha/k^\gamma$, it holds for all $i\in[n]$ that $\|\bar x^k - x_i^k\|^2=\cO(k^{-2\gamma})$. Hence, due to $\psi(\theta,\gamma) \leq 2\gamma$ and $\varphi(\theta,\gamma) \leq \gamma$, we have for all $i\in[n]$, 
\[\|\nabla f(x_i^k)\|^2 \leq \sL^2\|\bar x^k - x_i^k\|^2 + \|\nabla f(\bar x^k)\|^2  = \cO(k^{-\psi(\theta,\gamma)}) \; \text{and} \; \|x_i^k - x^*\| \leq \|x_i^k - \bar x^k \| + \|\bar x^k - x^*\|  = \cO(k^{-\varphi(\theta,\gamma)}). \]
Invoking \eqref{eq:dgd-Lipschitz}, we obtain \[|f(\bar x^k) - f(x_i^k)|\leq \tfrac{1}{2\sL}\max\{\|\nabla f(\bar x^k)\|^2,\|\nabla f(x_i^k)\|^2\} + \sL\|\bar x^k - x_i^k\|^2 = \cO(k^{-\psi(\theta,\gamma)}).\] Since $|f(\bar x^k)-f^*| = \cO(k^{-\psi(\theta,\gamma)})$, it follows from the triangle inequality that
 \[
 | f(x_i^k) - f^*| \leq |f(\bar x^k) - f(x_i^k)| + |f(\bar x^k)-f^*| = \cO(k^{-\psi(\theta,\gamma)})\quad \text{for all $i\in[n]$}.
 \]
 When $\theta = \tfrac12$, $\gamma=1$, $\alpha>8/\sC^2$, it holds that $\|\nabla f(x_i^k)\|^2 = \cO(k^{-2})$, $| f(x_i^k) - f^*|  = \cO(k^{-2})$ and $\|x_i^k - x^*\|=\cO(k^{-1})$ for all $i\in[n]$. Finally, when $\{\bar x^k\}_k$ is bounded, the result in (b) follows directly from \cref{thm:convergence rate} (b).
\end{proof}
As discussed in \cref{rem:convergence rate}, when the KL exponent $\theta\in[\frac12,1)$ is known, the optimal choice $\gamma^*$ of the step size parameter is $\gamma^*=\tfrac{\theta}{3\theta-1}$. In this case, we have for all agent $i\in[n]$
\[
\|\nabla f(x_i^k)\|^2 = \cO(k^{-\frac{1}{3\theta-1}}), \quad |f(x_i^k)-f^*| = \cO(k^{-\frac{1}{3\theta-1}}) \quad \text{and} \quad \|x_i^k - x^*\| = \cO(k^{-\frac{1-\theta}{3\theta-1}}),
\]
since the choice $\gamma = \gamma^*$ ensures $\gamma > \tfrac12$. As a consequence of \cref{thm:rate-dgd} (a), this choice of step sizes guarantees the iterate convergence of $\DGD$ for all agents $i\in[n]$ (without requiring additional boundedness of the average iterates). 

The successful application of our KL-based framework to establish the convergence of $\DGD$ suggests its broader applicability. Future work will explore its extension to analyze other variants of $\DGD$, such as decentralized $\SGD$ \cite{sundhar2010distributed} and 
decentralized random reshuffling method \cite{huang2023distributed}, in the nonconvex setting.
\subsection{Random Reshuffling}
\label{sec:rr}

The prior work \cite{li2023convergence} has shown that the random reshuffling ($\RR$) method can converge to a stationary point if the objective function satisfies the KL property. Nonetheless, the analysis in \cite{li2023convergence} relies on the ubiquitous bounded iterates assumption. In this section, we leverage our proposed KL-based framework to establish the same iterate convergence result under weaker assumptions. Moreover, our analysis further extends the existing results by providing convergence rates for both function values and the norm of the gradients.

The random reshuffling method performs the following steps in each iteration $k$: Let $\{\pi_1^k, \pi_2^k, \ldots, \pi_n^k\}$ denote a permutation of indices $[n]$ of the component functions and set
\[
x^{k+1} = x^k - \alpha_k\sum_{i=1}^n \nabla f_{\pi_i^k}(x_i^k),\quad \text{where}\quad x_1^k = x^k \quad \text{and} \quad x_{i+1}^k  =x_i^k - \alpha_k\nabla f_{\pi_i^k}(x_i^k).
\]
We now examine the update scheme of $\RR$ and derive key properties that will be instrumental in our subsequent analysis.
\begin{proposition}\label{lemma:rr}
Let \cref{as:func} hold and let $\{x^k\}_{k}$ be generated by $\RR$ with step sizes $\{\alpha_k\}_k$ satisfying 
	\begin{equation}
		\label{eq:rr-step-size-req}
		\alpha_k\in\Big(0,{\frac{1}{{\sqrt{2}\sf L}n}}\Big],\quad {\sum}_{k=1}^\infty \alpha_k=\infty \quad \text{and}\quad {\sum}_{k=1}^\infty \alpha_k^3 <\infty.
	\end{equation}
	Then, there exists $\sG>0$ such that for all $k\geq 1$, it holds that
\begin{enumerate}[label=\textup{(\alph*)},topsep=0ex,itemsep=0ex,partopsep=0ex, leftmargin = 25pt]
\item $f(x^{k+1})   \leq f(x^k) - \frac{n\alpha_k}{2}\|\nabla f(x^k)\|^2 + 2\sG\sL^3 n^3\alpha_k^3$.	
\item $\|x^{k+1} - x^k\| \leq n\alpha_k \|\nabla f(x^k)\| + 2\sG^{\frac12} \sL^{\frac32}  n^2 \alpha_k^2$.
\end{enumerate}
\end{proposition}
\begin{proof}
The proof is mostly based on \cite{li2023convergence}. Statement (a) restates a previously established result in \cite[Lemma 3.2]{li2023convergence}. Besides, it is inferred from \cite[Eq. (3.3) and (3.5)]{li2023convergence} that
\[
{\sum}_{i=1}^n \|x_i^k - x^k\|^2 \leq 4\sL\sG n^3 \alpha_k^2\quad \text{where} \quad \sG:=(f(x^1)-\bar f)\cdot \exp\big({\sum}_{k=1}^\infty 2\sL^3 n^3 \alpha_k^3\big).
\]
This implies \[\Big({\sum}_{i=1}^n \|x_i^k - x^k\|\Big)^2 \leq n{\sum}_{i=1}^n \|x_i^k - x^k\|^2 \leq 4\sL\sG n^4 \alpha_k^2.\] It follows from the update of $\RR$ and \cref{as:func} that
\begin{align*}
    \|x^{k+1} - x^k\| &= \alpha_k \Big\|{\sum}_{i=1}^n \nabla f_{\pi_i^k}(x_i^k)\Big\| \leq n\alpha_k \|\nabla f(x^k)\| + \alpha_k \Big\|{\sum}_{i=1}^n [\nabla f_{\pi_i^k}(x_i^k) - \nabla f_{\pi_i^k}(x^k)] \Big\|\\
    &\leq n\alpha_k \|\nabla f(x^k)\| + \sL \alpha_k {\sum}_{i=1}^n \|x_i^k - x^k\| \leq n\alpha_k \|\nabla f(x^k)\| + 2\sG^{\frac12} \sL^{\frac32}  n^2 \alpha_k^2.
\end{align*}
Hence, we have completed the proof of \cref{lemma:rr}.
\end{proof}
\cref{lemma:rr} establishes two key properties that align naturally with our framework due to their clear correspondence to conditions \ref{C1}--\ref{C2}. This enables us to directly apply \cref{thm:main} to recover the convergence results in \cite[Proposition 3.3 and Theorem 3.6]{li2023convergence}.
\begin{thm}
\label{thm:rr}
Let \cref{as:func} hold and let $\{x^k\}_{k}$ be generated by $\RR$ with the step sizes $\{\alpha_k\}_k$ satisfying
\begin{equation}
	\label{eq:step-size-rr}
	\alpha_k\in\Big(0,{\frac{1}{\sqrt{2}\sL n}}\Big],\quad {\sum}_{k=1}^{\infty}\alpha_k=\infty,\quad
	\text{and}\quad {\sum}_{k=1}^{\infty}\alpha_k\Big({\sum}_{i=k}^{\infty} \alpha_i^3\Big)^\mu<\infty\quad \text{for some }\mu\in(0,1).
\end{equation}
Then, $\lim_{k\to\infty}\|\nabla f(x^k)\|=0$ and $\{f(x^k)\}_k$ converges to some $f^*\in\R$. 
Moreover, if \cref{as:kl} holds, then $\{x^k\}_{k}$ either converges to some stationary point of $f$ or we have $\|x^k\| \to \infty$.  
\end{thm}

\begin{proof}
Utilizing the results in \cref{lemma:rr}, we can verify the conditions \ref{C1}--\ref{C2} by setting 
	\begin{equation*}
		%\label{eq:verify-DGD}
		I_k = k, \quad \beta_k = \alpha_k, \quad p_k = 2\sG\sL^3 n^3 \alpha_k^3, \quad q_k= 2\sG^{\frac12} \sL^{\frac32}  n^2 \alpha_k^2, \quad a_1 = \tfrac{n}{2}, \quad \text{and}\quad a_2 =n.
	\end{equation*}
In addition, the step sizes $\{\alpha_k\}_k$ satisfy \ref{C3}--\ref{C4}. \cref{thm:main} (a) implies that $\|\nabla f(x^k)\|\to 0$ and $f(x^k)\to f^*\in\R$ as $k$ tends to infinity.

When \cref{as:kl} holds,  \cref{thm:main} (b) guarantees the convergence result of $\{x^k\}_k$. 
\end{proof}
In \cite[Theorem 3.6]{li2023convergence}, boundedness of $\{x^k\}_k$ is required to achieve convergence. In contrast, \cref{thm:rr} only necessitates the weaker condition that $\acc(\{x^k\}_k) =\{x\in\Rn : \liminf_{k\to\infty}\|x^k-x\|=0\}$, is non-empty. 
Below, we show convergence rates for $\RR$ under polynomial step sizes by directly applying our framework \ref{R1}--\ref{R3}. 
\begin{thm}\label{thm:rr-rate}
	Let \cref{as:kl} and \ref{as:func} hold and let $\{x^k\}_{k}$ be generated by $\RR$ using step sizes $\{\alpha_k\}_k$ of the form $\alpha_k=\alpha/k^\gamma, \alpha>0, \gamma\in(0,1]$. 
	Then, the following statements hold.
\begin{enumerate}[label=\textup{(\alph*)},topsep=0ex,itemsep=0ex,partopsep=0ex, leftmargin = 25pt]
  \item If $\acc(\{x^k\}_k)$ is non-empty and $\gamma\in(\frac12,1)$, then $x^k \to x^*\in\crit(f)$ and $f(x^k)\to f(x^*) = f^*$ and we have
	 \[
	  |f(x^k)-f^*| = \cO(k^{-\psi(\theta,\gamma)}), \quad \|\nabla f(x^k)\|^2 = \cO(k^{-\psi(\theta,\gamma)}),\quad   \|x^k - x^*\| = \cO(k^{-\varphi(\theta,\gamma)}),\quad  \gamma\in(\tfrac{1}{2},1),
   \]
   where $\theta$ denotes the KL-exponent of $f$ at $x^*$ and
   \[\psi(\theta,\gamma):=	\min\{3\gamma-1, {\textstyle \frac{1-\gamma}{2\theta-1} } \},\quad \varphi(\theta,\gamma):=	\min\{2\gamma-1, {\textstyle \frac{(1-\gamma)(1-\theta)}{2\theta-1}}\}.
	 \]
Moreover, if $\theta=\half$, $\gamma=1$ and $\alpha > 8/(n\sC^2)$, then it holds that $|f(x^k)-f^*| = \cO(k^{-2})$, $\|\nabla f(x^k)\|^2 = \cO(k^{-2})$, and $\|x^k - x^*\| = \cO(k^{-1})$.
\item If $\{x^k\}_k$ is bounded, then $f(x^k)\to f^*$ and $|f(x^k)-f^*| = \cO(k^{-\psi(\theta,\gamma)})$ and $\|\nabla f(x^k)\|^2 = \cO(k^{-\psi(\theta,\gamma)})$ continue to hold in the case $\gamma\in(\frac{1}{3},\frac{1}{2}]$ with $\theta$ being the uniformized KL-exponent over the set $\acc(\{x^k\}_k)$.
 \end{enumerate}
\end{thm}
\begin{proof}
    The recursions in \cref{lemma:rr} and step sizes $\alpha_k=\alpha/k^\gamma$ imply that the conditions \ref{R1}--\ref{R3} hold with
    \[
    p=3,\quad q=2,\quad b=\tfrac{n}{2},\quad \tilde b=2\sG\sL^3n^3\quad \text{and} \quad c=\max\{n,2\sG^{\frac12} \sL^{\frac32}  n^2\}.
    \]
    When $\acc(\{x^k\}_k)$ is non-empty and $\gamma\in(\frac12,1]$, it follows from \cref{thm:convergence rate} (a) that $x^k \to x^*\in\crit(f)$ and $\{f(x^k)\}_k$ converges to $f^*:=f(x^*)$ with the rates 
     \[
	  |f(x^k)-f^*| = \cO(k^{-\psi(\theta,\gamma)}), \quad \|\nabla f(x^k)\|^2 = \cO(k^{-\psi(\theta,\gamma)}),\quad   \|x^k - x^*\| = \cO(k^{-\varphi(\theta,\gamma)}),\quad  \gamma\in(\tfrac{1}{2},1),
   \]
   where $\theta$ denotes the KL-exponent of $f$ at $x^*$ and
   \[\psi(\theta,\gamma):=	\min\{3\gamma-1, {\textstyle \frac{1-\gamma}{2\theta-1} } \},\quad \varphi(\theta,\gamma):=	\min\{2\gamma-1, {\textstyle \frac{(1-\gamma)(1-\theta)}{2\theta-1}}\}.
	 \]
Moreover, if $\theta=\half$, $\gamma=1$ and $\alpha > 8/(n\sC^2)$, then it holds that $|f(x^k)-f^*| = \cO(k^{-2})$, $\|\nabla f(x^k)\|^2 = \cO(k^{-2})$, and $\|x^k - x^*\| = \cO(k^{-1})$.
 Finally, when $\{x^k\}_k$ is bounded, the statement (b) follows directly from \cref{thm:convergence rate} (b). 
\end{proof}
Analogous to our findings for $\DGD$, the rates for $\RR$ in \cref{thm:rr-rate} can be optimized with respect to step size parameter $\gamma$ if the KL exponent $\theta$ is known. Specifically, if we choose $\alpha>8/(n\sC^2)$, $\gamma = \frac{\theta}{3\theta-1}$ and if $\acc(\{x^k\}_k)$ is non-empty, then $f(x^k)\to f^*\in\R$ and $x^k \to x^*\in\crit(f)$ with rates
 \[
\|\nabla f(x^k)\|^2 = \cO(k^{-\frac{1}{3\theta-1}}), \quad |f(x^k)-f^*| = \cO(k^{-\frac{1}{3\theta-1}}) \quad \text{and} \quad \|x^k - x^*\| = \cO(k^{-\frac{1-\theta}{3\theta-1}}).
\] 
Under the bounded iterates assumption, \cite[Theorem 3.10]{li2023convergence} established convergence rates for the iterates $\{x^k\}_k$ for $\RR$. Here, our results in \cref{thm:rr-rate} provide a more comprehensive understanding of $\RR$'s performance by further quantifying convergence rates for gradient norms and function values. 

Our analysis of $\RR$ in this section demonstrates the general effectiveness of our proposed framework in the context of stochastic nonconvex optimization. Given that $\RR$ serves as a fundamental building block for many other stochastic optimization methods, our framework holds potential for broader applications. In the following subsection, we will provide new convergence results for federated averaging method that incorporates the $\RR$ updates. 
\subsection{Federated Averaging}
\label{sec:fed}
The concept of federated learning was first proposed by McMahan et al. \cite{mcmahan2017communication} to train rich data separately by each client without centrally storing it. Moreover, McMahan et al. introduced federated averaging ($\Fed$, cf. \cite[Algorithm 1]{mcmahan2017communication}) which has become an highly influential method in large-scale (learning) problems \cite{he2019central,Li2020On,huang2024distributed,yun2021minibatch}. Most of existing analyses of $\Fed$ assume unbiasedness of the utilized stochastic gradients. However, practical implementations of $\Fed$ are based on \emph{without-replacement} sampling (i.e., shuffling) schemes to generate stochastic gradients\footnote{This can be observed in the 
\href{https://github.com/adap/flower/blob/main/examples/quickstart-pytorch/client.py}{example} provided by the well-known federated learning framework \href{https://flower.ai/}{Flower}.}. This type of sampling does not satisfy the unbiasedness assumption. 

For the sake of brevity, in the following discussions, $\Fed$ will refer to the $\Fed$ algorithm with shuffling. In \cite{malinovsky2023server,mishchenko2022proximal}, it was shown that $\Fed$ converges to a neighborhood of the solution in strongly convex settings and when constant step sizes are used. Under a full device participation and certain deviation bounds and utilizing the (global) Polyak-\L ojasiewicz (PL) condition, \cite{yun2021minibatch} have shown that $\Fed$ can achieve the rate $f(x^k) - f^*=\cO(k^{-2})$. To the best of our knowledge, convergence guarantees of the form $\|\nabla f(x^k)\|\to 0$ and $x^k \to x^*\in \crit(f)$ for $\Fed$ remain unknown in the nonconvex case. In this subsection, we will address this gap by providing such convergence results for $\Fed$ under mild assumptions. In addition, we derive convergence rates for $\{x^k\}_k$.

Our analysis focuses on a simplified setting involving full client participation, an equal number of local updates, and an equal number of component functions across clients. We consider the following formulation of federated learning:
\begin{align*}
	f(x):=\frac1n\sum_{t=1}^{n}f_t(x),\quad \text{where each $f_t$ has finite-sum structure}\quad f_t(x):=\frac{1}{m}\sum_{j=1}^m h^t_{j}(x),
\end{align*}
where $n$ is the number of clients and $m$ is the number of component functions for each client $i$. Here, we work with Lipschitz continuity assumption for each function $h^t_j:\Rn\to\R$ -- analogous to \cref{as:func}.
\begin{assumption}
	\label{as:fed}
	Every component function $h^t_j:\Rn \to \R$, $t\in[n]$,$j\in[m]$, is ${\sf L}$-smooth and bounded from below by $\bar f$.
\end{assumption}
As noted in \cref{as:func-1}, the Lipschitz smoothness provides a useful upper bound for all $t\in[n]$ and $j\in[m]$:
\begin{equation}
	\label{eq:L-smooth-Fed}
	\|\nabla h^t_j(x)\|^2 \leq 2\sL  (h^t_j(x) - \bar f),\quad \forall~x\in\Rn.
\end{equation}
Let $E\in \N_+$ denote the number of local epochs. The main update of $\Fed$ at iteration $k$ is given by:
\begin{center}
    \begin{tabular}{|p{7cm}|c|}
    \hline
    & \\[-1mm]
\centering \underline{\textbf{Each client $t\in[n]$}} &\underline{\textbf{Central server}} \\[1mm]
 $y_{t,k}^{1,1} = x^k$ & \\[1mm]
\textbf{For $i=1,\ldots,E$} &\\
\quad generate permutation $\{\pi^{i,1}_{t},\ldots,\pi^{i,m}_{t}\}$ of $[m]$ &\\[-2mm]
\quad \textbf{For $j=1,\ldots,m$} & $ \begin{aligned}
    x^{k+1} = \frac{1}{n} \sum_{t=1}^n\; x_t^{k+1}
\end{aligned}$\\[-2mm]
\qquad $y_{t,k}^{i,j+1} =  y_{t,k}^{i,j} - \alpha_k \nabla h_{\pi^{i,j}_{t}}(y_{t,k}^{i,j})$ & \\[2mm]
\quad  \textbf{End For} &\\[1mm]
\quad $y_{t,k}^{i+1,1} = y_{t,k}^{i,m+1}$&\\[1mm]
\textbf{End For} &\\
$x_t^{k+1} = y_{t,k}^{E+1,1}$ &\\[1mm]
\hline
\end{tabular}
\end{center}
Based on the update of $\Fed$, we obtain the following algorithmic bounds. The proof can be found in \cref{proof:FedAvg}.
\begin{proposition}
		\label{lemma:FedAvg}
	Assume \cref{as:fed} holds. Let the sequence $\{x^k\}_k$ be generated by $\Fed$ with step sizes $\{\alpha_k\}_k$ satisfying $0<\alpha_k \leq \frac{1}{2mE\sL}$ and $\sum_{k=1}^\infty \alpha_k^3 <\infty$. Then, there exists $\sG>0$ such that for all $k\geq 1$, the following statements hold:
\begin{enumerate}[label=\textup{(\alph*)},topsep=0ex,itemsep=0ex,partopsep=0ex, leftmargin = 25pt]
\item $f(x^{k+1})   \leq f(x^k) - \frac{Em\alpha_k}{2}\|\nabla f(x^k)\|^2 + \sG  \alpha_k^3$	
\item $\|x^k-x^{k+1}\| \leq  Em \cdot\alpha_k \|\nabla f(x^k)\| + \sqrt{2m\sG E} \cdot \alpha_k^2$
\end{enumerate}
\end{proposition}

Given the similarities of the derived bounds in \cref{lemma:rr} and \ref{lemma:FedAvg}, we can establish analogous results for $\Fed$ by mirroring the verification process for $\RR$.
\begin{thm}
\label{thm:fed}
Let \cref{as:fed} hold and let $\{x^k\}_{k}$ be generated by $\Fed$ with the step sizes $\{\alpha_k\}_k$ satisfying
\begin{equation}
	\label{eq:step-size-fed}
	\alpha_k\in\Big(0,{\frac{1}{2mE\sL}}\Big],\quad {\sum}_{k=1}^{\infty}\alpha_k=\infty,\quad
	\text{and}\quad {\sum}_{k=1}^{\infty}\alpha_k\Big({\sum}_{i=k}^{\infty} \alpha_i^3\Big)^\mu<\infty\quad \text{for some }\mu\in(0,1).
\end{equation}
Then, $\lim_{k\to\infty}\|\nabla f(x^k)\|=0$ and $\{f(x^k)\}_k$ converges to some $f^*\in\R$. 
Moreover, if \cref{as:kl} holds, then either $\{x^k\}_{k}$ converges to some stationary point of $f$ or we have $\|x^k\| \to \infty$.  
\end{thm}

To our knowledge, the only prior work investigating the asymptotic behavior of $\Fed$ applied to nonconvex objectives is \cite{huang2024distributed}. Under the assumptions of full device participation and unbiased stochastic gradients, Huang et al. have shown $\|\nabla f(x^k)\|\to0$ almost surely  \cite[Theorem 2]{huang2024distributed}. By contrast, \cref{thm:fed} establishes $\|\nabla f(x^k)\|\to0$ for $\Fed$ when shuffling is used and without requiring the ubiquitous unbiasedness assumption \cite{he2019central,huang2024distributed,Li2020On}.

When $f$ satisfies the KL property and the accumulation points set $\acc(\{x^k\}_k)$ is non-empty, \cref{thm:fed} further implies that $\{x^k\}_k$ converges to some stationary point $x^*\in\crit(f)$. Building upon this, we then proceed to derive convergence rates for $\Fed$.  
\begin{thm}\label{thm:fed-rate}
	Let \cref{as:kl} and \ref{as:fed} hold and let $\{x^k\}_{k}$ be generated by $\Fed$ with step sizes $\{\alpha_k\}_k$ of the form $\alpha_k=\alpha/k^\gamma, \alpha>0, \gamma\in(0,1]$. 
	Then, the following statements hold.
 
	 \begin{enumerate}[label=\textup{(\alph*)},topsep=0ex,itemsep=0ex,partopsep=0ex, leftmargin = 25pt]
    \item If $\acc(\{x^k\}_k)$ is non-empty and $\gamma\in(\frac12,1)$, then $x^k \to x^*\in\crit(f)$ and $f(x^k)\to f(x^*) = f^*$ and we have
	 \[
	  |f(x^k)-f^*| = \cO(k^{-\psi(\theta,\gamma)}), \quad \|\nabla f(x^k)\|^2 = \cO(k^{-\psi(\theta,\gamma)}),\quad   \|x^k - x^*\| = \cO(k^{-\varphi(\theta,\gamma)}),\quad  \gamma\in(\tfrac{1}{2},1),
   \]
   where $\theta$ denotes the KL-exponent of $f$ at $x^*$ and
   \[\psi(\theta,\gamma):=	\min\{3\gamma-1, {\textstyle \frac{1-\gamma}{2\theta-1} } \},\quad \varphi(\theta,\gamma):=	\min\{2\gamma-1, {\textstyle \frac{(1-\gamma)(1-\theta)}{2\theta-1}}\}.
	 \]
Moreover, if $\theta=\half$, $\gamma=1$ and $\alpha > 8/(mE\sC^2)$, then it holds that $|f(x^k)-f^*| = \cO(k^{-2})$, $\|\nabla f(x^k)\|^2 = \cO(k^{-2})$, and $\|x^k - x^*\| = \cO(k^{-1})$.
\item If $\{x^k\}_k$ is bounded, then $f(x^k)\to f^*$ and $|f(x^k)-f^*| = \cO(k^{-\psi(\theta,\gamma)})$ and $\|\nabla f(x^k)\|^2 = \cO(k^{-\psi(\theta,\gamma)})$ continue to hold in the case $\gamma\in(\frac{1}{3},\frac{1}{2}]$ with $\theta$ being the uniformized KL-exponent over the set $\acc(\{x^k\}_k)$.
 \end{enumerate}
\end{thm}
Similar to $\RR$ and $\DGD$, when $\theta\in[\frac12,1)$ is known, the rates for $\Fed$ in \cref{thm:fed-rate} can be optimized to
 \[
\|\nabla f(x^k)\|^2 = \cO(k^{-\frac{1}{3\theta-1}}), \quad |f(x^k)-f^*| = \cO(k^{-\frac{1}{3\theta-1}}), \quad \text{and} \quad \|x^k - x^*\| = \cO(k^{-\frac{1-\theta}{3\theta-1}}).
\] 
In summary, this subsection establishes iterate convergence and derives the corresponding rates for $\Fed$ under mild Lipschitz smoothness and KL-based assumptions. In particular, when $\theta = \frac12$, \cref{thm:fed-rate} recovers the function rate $f(x^k)-f^* = \cO(k^{-2})$ shown in \cite[Theorem 1]{yun2021minibatch} requiring the more restrictive PL condition. Furthermore, \cref{thm:fed-rate} provides the additional rates $\|\nabla f(x^k)\|^2 = \cO(k^{-2})$ and $\|x^k-x^*\|=\cO(k^{-1})$ in this case.

\section{Conclusion} 
We propose a novel KL-based analysis framework that is applicable to algorithms that do not necessarily possess a sufficient descent property. This framework allows us to study a broader class of optimization methods, such as stochastic and distributed algorithms. Leveraging the framework, we provide new convergence results for decentralized gradient ($\DGD$) and federated averaging ($\Fed$) methods and recover existing results for $\SGD$ and random reshuffling ($\RR$) without requiring an a priori boundedness condition on the iterates.

As an additional by-product, we also provide a stream-lined way to quantify the convergence rates of an algorithm if it and the utilized step sizes fit in a specialized form of our proposed framework. This allows providing new rates for the function values and gradient norms for $\RR$. In addition, this specialized framework facilitates the derivation of convergence rates for the iterates and gradient norms for $\DGD$ and $\Fed$. These results appear to be new -- even if we would work with a significantly more restrictive, global Polyak-{\L}ojasiewicz assumption.   

\bibliographystyle{siam}
\bibliography{references} 
\appendix
\section{Proof of Main Convergence Results} \label{proof:main}
\subsection{Proof of \cref{thm:main} (a)}\label{proof:prop:weak-con}
\begin{proof}
Without loss of generality, we discard all iterates up to the $K$ iterate and relabel the
sequence, and then \ref{C1} and \ref{C2} hold explicitly on the new sequence. Based on \ref{C1}, by defining $u_k:={\sum}_{i=k}^{\infty}\, p_i$, 
	\begin{equation}
		\label{eq:prop-1}
		f(x^{I_{k+1}}) + u_{k+1} \leq f(x^{I_k}) + u_{k} - a_1 \beta_k\|\nabla f(x^{I_k})\|^2.
	\end{equation}
	Let us notice that $\{p_k\}_k$ is summable and, hence, $\{u_k\}_k$ is finite and non-increasing with $u_k\to0$. Next, consider the sequence $\{f(x^{I_k})+u_k\}_k$ which is nonincreasing and bounded from below, so there must be a constant $f^*$ such that $\lim_{k\to\infty}f(x^{I_k})+u_k=f^*$. Since $\lim_{k\to\infty}u_k=0$, we conclude that $\lim_{k\to\infty} f(x^{I_k})=f^*$. Moreover, there exists an upper bound ${\sG}:= f(x^{I_0}) + u_0 -\bar f$ for the sequence $\{f(x^{I_k})-\bar f\}_{k\in\N}$. 
	
	Then, unfolding the recursion \cref{eq:prop-1} by summing over $k=1,\cdots,M$, and letting $M\to\infty$, this yields
	\begin{equation}
		\label{eq:prop-2}
		\infty > a_1^{-1}{\sG} \geq {\sum}_{k=1}^{\infty} \beta_k \|\nabla f(x^{I_k})\|^2 =:{\sum}_{k=1}^{\infty} \beta_k F_k^2.
	\end{equation}
	Notice that $\sum_{k=1}^{\infty}\beta_k=\infty$, a direct result from \cref{eq:prop-2} is $\liminf_{k\to\infty} F_k=0$. However, it still remains to be shown that $\lim\limits_{k\to\infty}F_k=0$.
	Let us assume on the contrary that the sequence $\{F_k\}_k$ does not converge to zero. Then, there exists $\varepsilon>0$ and two infinite subsequences $\{t_j\}$ and $\{\ell_j\}$ such that $t_j < \ell_j < t_{j+1}$,
	\begin{equation}
		\label{eq:prop-3}
		F_{t_j} \geq 2\varepsilon,\quad F_{\ell_j} <\varepsilon,\quad\text{and}\quad \varepsilon \leq F_t<2\varepsilon
	\end{equation}
for all $k=t_j+1,\cdots,\ell_j-1$.
Combing \cref{eq:prop-3} with \cref{eq:prop-2} yields
\[\infty>{\sum}_{k=1}^{\infty}\beta_k F_k^2\geq \varepsilon^2 {\sum}_{j=1}^{\infty}{\sum}_{k=t_j}^{\ell_j-1}\beta_k=:\varepsilon^2 {\sum}_{j=1}^{\infty}\zeta_j,\]
which implies
\begin{equation}
	\label{eq:prop-4}
	\lim\limits_{j\to\infty}\zeta_j=0.
\end{equation}
According to \ref{C2}, we obtain the following inequality for $k \geq 1$, %large enough such that $\alpha_k<1$,
\begin{equation}
	\label{eq:prop-4-1}
	\begin{aligned}
		\|x^{I_{k+1}}-x^{I_k}\| &\leq \max_{I_{k} < i  \leq I_{k+1}} \|x^i - x^{I_{t_k}} \| \leq a_2\beta_k \|\nabla f(x^{I_k})\| + q_k \\& \leq a_2\beta_k \sqrt{2{\sL}(f(x^{I_k}) - \bar f)} + q_k \leq a_2\sqrt{2\sL\sG}\beta_k + q_k,
	\end{aligned}
\end{equation}
where the third inequality is due to \eqref{eq:L-smooth}. Then, we apply Cauchy-Schwartz inequality,
\begin{align*}
	\|x^{I_{\ell_j}} - x^{I_{t_j}}\| &\leq {\sum}_{k=t_j}^{\ell_j-1}\sqrt{\beta_k}\left[\frac{\|x^{I_{k+1}}-x^{I_k}\|}{\sqrt{\beta_k}}\right]\leq \sqrt{\zeta_j} \left[{\sum}_{k=t_j}^{\ell_j-1}\beta_k^{-1}\|x^{I_{k+1}}-x^{I_k}\|^2 \right]^{\frac12}\\
	&\leq \sqrt{\zeta_j} \left[{\sum}_{k=t_j}^{\ell_j-1}4a^2_2{\sL \sG}\beta_k + 2\beta_k^{-1}q_k^2 \right]^{\frac12}=\sqrt{\zeta_j} \left[4a^2_2{\sL \sG}\zeta_j +  2{\sum}_{k=t_j}^{\ell_j-1}\beta_k^{-1}q_k^2 \right]^{\frac12}
\end{align*}
On the one hand, upon taking the limit $j\to\infty$ in the above inequality, together with \cref{eq:prop-4} and $\sum_{k=1}^{\infty} \beta_k^{-1}q_k^2 <\infty$ (implied by \ref{C3}), we have
\begin{equation}
	\label{eq:prop-5}
	\lim\limits_{j\to\infty}\|x^{I_{\ell_j}}-x^{I_{t_j}}\|=0.
\end{equation}
On the other hand, combing \cref{eq:prop-3}, the triangle inequality, and the Lipschitz continuity of $\nabla f$, we have
\begin{equation}
	\label{eq:prop-6}
	\varepsilon \leq | F_{\ell_j}-F_{t_j} | \leq \|\nabla f(x^{I_{\ell_j}}) - \nabla f(x^{I_{t_j}})\| \leq {\sL}\| x^{I_{\ell_j}}- x^{I_{t_j}}\|. 
\end{equation}
We reach a contradiction to \cref{eq:prop-5} by taking $j\to\infty$ in \cref{eq:prop-6}. Consequently, we conclude that $\lim\limits_{k\to\infty}\|\nabla f(x^{I_k})\|=0$. To establish $\lim\limits_{k\to\infty}\|\nabla f(x^{k})\|=0$, we need to explore the relation between $\|\nabla f(x^{k})\|$ and $\|\nabla f(x^{I_k})\|$. For any index $j\in\N$, there exists $k$ such that $I_k < j \leq I_{k+1}$. Then, by $\sL$-continuity of $\nabla f$ and the estimate \eqref{eq:prop-4-1}, 
\begin{align*}
	\|\nabla f(x^j)\| &\leq \|\nabla f(x^j) - \nabla f(x^{I_k})\| + \|\nabla f(x^{I_k})\| \leq {\sL}\|x^j - x^{I_k}\| + \|\nabla f(x^{I_k})\| \\ &\leq (1+{\sL}a_2 \beta_k ) \|\nabla f(x^{I_k})\| + {\sL}q_k \leq (1+ {\sL}a_2\bar \beta) \|\nabla f(x^{I_k})\| + {\sL} q_k,
\end{align*}
where the last inequality is due to $\beta_k \leq \bar \beta $. Let us notice that $\|\nabla f(x^{I_k})\|\to0$ and  $q_k\to0$ as $k$ tends to infinity, we have shown that $\lim\limits_{k\to\infty}\|\nabla f(x^{k})\|=0$. Note that $\sum_{k=1}^{\infty} \bar \beta^{-1}q_k^2  < \sum_{k=1}^{\infty} \beta_k^{-1}q_k^2 <\infty$ implies $q_k\to0$. Invoking the $\sL$-continuity of $\nabla f$ (cf. \cite[Lemma 1.2.3]{nesterov2018lectures}), we have
\[
|f(x^i)-f(x^{I_k}) - \langle \nabla f(x^{I_k}),x^i-x^{I_k}\rangle| \leq \frac{\sL}{2} \|x^i-x^{I_k}\|^2,\quad \forall\; i=I_k,I_k+1,\dots,I_{k+1}.
\]
Thanks to the facts  $\max_{I_{k} < i  \leq I_{{k+1}}} \|x^i - x^{I_{k}} \| \leq a_2\beta_k \|\nabla f(x^{I_k})\| + q_k \to 0$ and $\nabla f(x^k)\to0$, we can show $\max_{I_{k} < i  \leq I_{{k+1}}} |f(x^i)-f(x^{I_k}) | \to 0$, which, together with $f(x^{I_k})\to f^*$, establishes $f(x^k) \to f^*$ as $k\to\infty$. 
\end{proof}
\subsection{Proof of \cref{thm:main} (b)}\label{proof:strong-con}
We define the set of accumulation points of the sequence $\{x^{I_k}\}_k$ as
\begin{equation}
	\label{accumulation-point-0}
 \acc(\{x^{I_k}\}_k):=
    \{x\in\Rn : \liminf_{k\to\infty}\|x^{I_k}-x\|=0\}.
\end{equation}
It is inferred from \cref{thm:main} (a) that $\|\nabla f(x^{I_k})\| \to 0$, which indicates that every accumulation point of $\{x^{I_k}\}_k$ is a stationary point. Consequently, $\acc(\{x^{I_k}\}_k) \subseteq\crit(f)$. 

To proceed towards the iterate convergence \cref{thm:main} (b), we introduce a crucial technical result \cref{lem:dk-iter}. This lemma provides a KL-based bound that plays a central role in our subsequent analysis. The proof of \cref{lem:dk-iter} is deferred to \cref{proof:lem:dk-iter} for clarity of presentation.
\begin{lemma} \label{lem:dk-iter}
    Suppose that the conditions stated in \cref{thm:main} (a) hold. If there exists $x^* \in \acc(\{x^{I_k}\}_k) \subseteq\crit(f)$ and $j \in \mathbb{N}_+$ such that $x^{I_j} \in U(x^*)$ and $|f(x^{I_j}) - f(x^*)| < \min\{1, \eta\}$, then
    \begin{equation}
		\label{eq:est-0}
		d_j  \leq  \frac{\sqrt2a_2}{a_1}(\Delta_j-\Delta_{j+1}) + \sqrt2 \sC a_2 \beta_j v_j^{\vartheta},
    \end{equation}
    where $d_j:= \max_{I_j < i \leq I_{j+1}}\|x^{i}-x^{I_j}\|$, $\Delta_j:=\frac{1}{\sC(1-\vartheta)}[f(x^{I_j})- f(x^*) + v_j]^{1-\vartheta}$ and $v_j:= \sum_{i=j}^{\infty}p_i + \frac{a_1}{2a^2_2}\sum_{i=j}^{\infty}\beta_i^{-1}q_i^2$.
\end{lemma}

\emph{Proof of \cref{thm:main} (b).}
If $\|x^k\| \nrightarrow \infty$, then $\{x^k\}_k$ has at least one accumulation point $x^*\in\Rn$. Since $\nabla f(x^k)\to0$, we conclude that $x^*\in\crit(f)$. By \cref{as:kl}, the following KL inequality holds at $x^*$, i.e., 
\begin{equation*}
    \|\nabla f(x)\| \geq \sC|f(x)-f(x^*)|^\vartheta,\quad \text{where} \quad \vartheta\in[\theta,1),
\end{equation*}
holds for all $x \in U(x^*) \cap \{x \in \Rn: 0 < |f(x) - f(x^*)| < \min\{1,\eta\}\}$. 
Moreover, from the proof of \cref{thm:main} (a), we have shown that $d_k = \max_{I_k < i \leq I_{k+1}} \|x^i - x^{I_k}\| \to 0$. Hence, there exists a subsequence $\{x^{I_{\ell_k}}\}_k \subseteq \{x^{I_k}\}_k$ converging to $x^*$. Since $f(x^k)\to f^*$ and $f(x^{I_{\ell_k}})\to f(x^*)$ (due to continuity of $f$), we conclude that $f(x^*)=f^*$, $\Delta_k \to 0$ ($v_k \to 0$ by \ref{C3}), and there is $K_f \geq K$ such that 
\begin{equation}
    \label{eq:thm:kl-1}
     |f(x^k) - f(x^*)| < \min\{1,\eta\}\quad \text{for all}\;k \geq K_f.
\end{equation}
\noindent Now, invoking subadditivity of $x^\vartheta$ when $\vartheta \in [0, 1)$ and \ref{C4},
we obtain
\[
    {\sum}_{i=1}^\infty \beta_i v_i^\vartheta \leq
    {\sum}_{i=1}^\infty \beta_i \big({\sum}_{k=i}^\infty p_k\big)^\vartheta + 
    \Big(\frac{a_1}{2a_2^2}\Big)^\vartheta {\sum}_{i=1}^\infty \beta_i \big({\sum}_{k=i}^\infty \beta_k^{-1}q_k^2\big)^\vartheta < \infty.
\]
Hence, combining the above, for any given $\rho>0$ fulfilling $\cB(x^*,\rho) \subseteq U(x^*)$,  there is $t \geq K_f$ such that
\begin{equation}
    \label{eq:thm-kl-rho}
    \|x^{I_{t}} - x^*\| + \frac{\sqrt2 a_2}{a_1}\Delta_{t} + \sqrt2 \sC a_2{\sum}_{i=t}^\infty \beta_i v_i^\vartheta < \rho.
\end{equation}

The main component of this proof is to show that the following statements are true for all $k \geq t$: 
    \begin{enumerate}[label=(\alph*),topsep=4pt,itemsep=0ex,partopsep=0ex, after=\vspace{6pt}]
    \item $x^{I_k} \in \mathcal{B}(x^*,\rho)$ and $|f(x^{I_k}) - f^*| < \min\{1,\eta\}$.
    \vspace{-1mm}
    \item $\sum_{i=t}^{k}d_i  \leq \frac{\sqrt2 a_2}{a_1}(\Delta_{t} - \Delta_{k+1}) + \sqrt2 \sC a_2 {\sum}_{i=t}^k \beta_i v_i^\vartheta.$
    \end{enumerate}
    We prove these statements by induction. Clearly, statements (a) and (b) hold for $k=t$ by \cref{lem:dk-iter}. Let us assume there is $m > t$ such that the statements (a) and (b) are valid for $k=m$. We now turn to $k=m+1$. It is inferred from \eqref{eq:thm:kl-1} that $|f(x^{I_{m+1}}) - f^*| < \min\{1,\eta\}$. We now show that $x^{I_{m+1}}\in \mathcal{B}(x^*,\rho)$. Using triangle inequality and statement (b), we obtain 
	\begin{align*}
		\|x^{I_{m+1}} - x^*\| &\leq \|x^{I_{m+1}} - x^{I_{m}}\| + \|x^{I_{m}} - x^{I_{t}} \| +\|x^{I_{t}} - x^*\| \leq \|x^{I_{t}} - x^*\| +  {\sum}_{i=t}^{m}d_i \\ &\leq \|x^{I_{t}} - x^*\| + \frac{\sqrt2 a_2}{a_1}(\Delta_{t} - \Delta_{k+1}) + \sqrt2 \sC a_2 {\sum}_{i=t}^m \beta_i v_i^\vartheta < \rho,
	\end{align*}
	where the last inequality follows from \eqref{eq:thm-kl-rho} and $\Delta_{k} \geq 0$ for all $k\geq 1$. This accomplishes the statement (a) for $k={m+1}$, implying that $x^{I_{m+1}} \in U(x^*)$ and $|f(x^{I_{m+1}}) - f^*| < \min\{1,\eta\}$. Hence, \Cref{lem:dk-iter} is applicable for $j=m+1$, i.e., we have
	\[d_{m+1}  \leq  \frac{\sqrt2a_2}{a_1}(\Delta_{m+1}-\Delta_{m+2}) + \sqrt2 \sC a_2 \beta_{m+1} v_{m+1}^{\vartheta}.\]
	Combining this inequality with the bound (when $k=m$) in (b) yields
	 \[{\sum}_{i=t}^{m+1}d_i  \leq \frac{\sqrt2 a_2}{a_1}(\Delta_{t} - \Delta_{m+2}) + \sqrt2 \sC a_2 {\sum}_{i=t}^{m+1} \beta_i v_i^\vartheta, \]
	 which indicates that (b) is also valid for $k=m+1$. 
	 Therefore, we show the statements (a) and (b) are valid for all $k \geq t$. It then follows from (b) and \eqref{eq:thm-kl-rho} that
	 \[
	 {\sum}_{i=t}^{\infty}d_i  \leq \frac{\sqrt2 a_2}{a_1}\Delta_{t} + \sqrt2 \sC a_2 {\sum}_{i=t}^\infty \beta_i v_i^\vartheta < \rho <\infty.
	 \]
  Then, for any given $\varepsilon>0$, recall that $d_k=\max_{I_k < i \leq I_{k+1}}\|x^{i}-x^{I_k}\|$, there exists an integer $k_1\geq t$ such that
	\begin{equation}\label{eq:est-6}
		{\sum}_{k=k_1}^{\infty}\|x^{I_{k+1}}-x^{I_k}\| <\frac{\varepsilon}{3}\quad \text{and}\quad \max_{I_k < i \leq I_{k+1}}\|x^{i}-x^{I_k}\| < \frac{\varepsilon}{3}\quad \text{for all }k\geq k_1.
	\end{equation} 
 Hence, for arbitrary integers $m,n$ satisfying $I_{k_1}<m<n$, there exist $k_2$ and $k_3$ (potentially fulfilling $k_1\leq k_2 \leq k_3$) such that $I_{k_2} < m \leq I_{k_2+1}$, $I_{k_3} < n \leq I_{k_3+1}$,  and thus
	\begin{align*}
		\|x^m-x^n\| \leq \|x^m-x^{I_{k_2}}\| + \|x^n-x^{I_{k_3}}\| + {\sum}_{k=k_2}^{k_3}\|x^{I_{k+1}}-x^{I_k}\| <  \frac{\varepsilon}{3} + \frac{\varepsilon}{3} +\frac{\varepsilon}{3} =\varepsilon.
	\end{align*}
 Hence, the sequence $\{x^k\}_k$ is Cauchy, which, together with \cref{thm:main} (a), implies $\{x^k\}_k$ converges to some stationary point of the objective function $f$. \hfill \qedsymbol{}
%\end{proof}

\subsection{Proof of \cref{lem:dk-iter}} \label{proof:lem:dk-iter}
\begin{proof}
    Taking square on both sides of \ref{C2}, invoking the inequality $(a+b)^2\leq 2a^2 + 2b^2$, and multiplying $\beta_k^{-1}$ yield
	\[\beta_k^{-1}d_k^2 \leq 2a^2_2\beta_k \|\nabla f(x^{I_k})\|^2 + 2 \beta_k^{-1}q_k^2.\]
	Using the former relation, we further rewrite \ref{C1} to create a suitable estimate that copes with the KL inequality-based analysis, i.e., 
	\begin{align*}
		f({x}^{I_{k+1}}) &\leq f({x}^{I_k}) + p_k -a_1\beta_k\|\nabla f(x^{I_k})\|^2 \\&\leq f({x}^{I_k}) + p_k - \frac{a_1\beta_k}{2}\|\nabla f(x^{I_k})\|^2 - \frac{a_1d_k^2}{4a^2_2\beta_k} + \frac{a_1q_k^2}{2a^2_2\beta_k}.
	\end{align*}
    Now, notice that the sequence $\{v_k\}_k$ is well defined thanks to \ref{C4}. Then,
	\begin{equation}
		\label{eq:est-1}
	f({x}^{I_{k+1}}) + v_{k+1} \leq f({x}^{I_k}) + v_k - \frac{a_1}{2}\beta_k\|\nabla f(x^{I_k})\|^2 - \frac{a_1}{4a^2_2}\beta_k^{-1}d_k^2.
	\end{equation}
    The inequality \cref{eq:est-1} infers that the sequence $\{f(x^{I_k})+v_k\}_k$ is non-increasing. In addition, $f(x^{I_k})\to f^*$ for some $f^*\in\R$ owing to $v_k\to0$ (implied by \ref{C3}). Since $x^* \in \acc(\{x^{I_k}\}_k)$, by continuity of $f$, we have $f^* = f(x^*)$.
    Due to $x^{I_j} \in U(x^*)$ and $|f(x^{I_j}) - f(x^*)| < \min\{1, \eta\}$, by the KL property, it holds that for all $\vartheta \in[\theta,1)$, 
    \begin{equation} \label{eq:kl-key-eq1}
         \|\nabla f(x) \| \geq \sC |f(x) - f(x^*)|^\theta \geq \sC |f(x) - f(x^*)|^{\vartheta}.
    \end{equation}
    Let us define $\varrho(s):=\frac{1}{\sC (1-\vartheta)}s^{1-\vartheta}$ (so $[\varrho'(s)]^{-1} = \sC s^\vartheta$) and the sequence
    \begin{equation}
		\label{eq:def-klfun}
		\Delta_k:=\varrho(f(x^{I_k})- f(x^*) + v_k).
    \end{equation}
    Note that $\Delta_k$ is well-defined since $f(x^{I_k}) + v_k\downarrow f(x^*)$. Based on \eqref{eq:kl-key-eq1} and $\vartheta \in [0, 1)$, we have
    \begin{equation}
		\label{eq:est-2}
		\begin{aligned}
			\Delta_j - \Delta_{j+1} &\geq \varrho^\prime(f(x^{I_j}) - f^* + v_j)\left[f(x^{I_j})+v_j - f(x^{I_{j+1}}) - v_{j+1} \right]\\
			& \geq \varrho^\prime(|f(x^{I_j}) - f^*| + v_j)\left[f(x^{I_j})+v_j - f(x^{I_{j+1}}) - v_{j+1} \right]\\
			&\geq \varrho^\prime(|f(x^{I_j}) - f^*| + v_j)\left[ \frac{a_1\beta_j}{2}\|\nabla f(x^{I_j})\|^2 + \frac{a_1d_j^2}{4a^2_2\beta_j} \right] \\&\geq \frac{1}{[\varrho^\prime(|f(x^{I_j}) - f^*|)]^{-1} + [\varrho^\prime(v_j)]^{-1} } \left[\frac{a_1\beta_j}{2}\|\nabla f(x^{I_j})\|^2 + \frac{a_1d_j^2}{4a_2^2\beta_j}\right]  ,
		\end{aligned}
    \end{equation}
    where the first inequality uses the concavity of $\varrho$, the second inequality uses non-increasing property of $\varrho'$, while the third inequality holds due to \cref{eq:est-1} and the last inequality uses $(x+y)^{\vartheta} \leq x^{\vartheta} + y^{\vartheta}$ for all $x,y\geq0 $ and $\vartheta\in[0,1)$.  Rearranging \eqref{eq:est-2} and applying \eqref{eq:kl-key-eq1} yield
    \[
	\frac{a_1\beta_j}{2}\|\nabla f(x^{I_j})\|^2 + \frac{a_1d_j^2}{4a^2_2\beta_j} \leq (\Delta_j-\Delta_{j+1}) [\|\nabla f(x^{I_j})\| + [\varrho^\prime(v_j)]^{-1}] = (\Delta_j-\Delta_{j+1}) [\|\nabla f(x^{I_j})\| + \sC v_j^{\vartheta}].
    \]
    Multiplying both sides by $a_1^{-1}\beta_j$ and using the inequality $(a+b)^2 \leq 2a^2 + 2b^2$, we have
    \[
    \left[\beta_j\|\nabla f(x^{I_j})\|/2 + a_2^{-1}d_j/(2\sqrt2)\right]^2
    \leq a_1^{-1}(\Delta_j - \Delta_{j+1}) [\beta_j\|\nabla f(x^{I_j})\| + \sC\beta_j v_j^{\vartheta}].
    \]
    Taking the square root of the former estimate and utilizing estimate $\sqrt{ab}\leq \frac{1}{2}a + \frac{1}{2} b$, we obtain
    \begin{align*}
	\frac{1}{2}\beta_j\|\nabla f(x^{I_j})\| + \frac{1}{2\sqrt2 a_2}d_j
	&\leq \sqrt{a_1^{-1} (\Delta_j - \Delta_{j+1})[\beta_j\|\nabla f(x^{I_j})\| + \sC\beta_j v_j^{\vartheta}]}\\
	&\leq\frac{1}{2a_1}(\Delta_j-\Delta_{j+1}) + \frac12\beta_j\|\nabla f(x^{I_j})\| + \frac{\sC\beta_j v_j^{\vartheta}}{2},\quad \forall k\geq k_0.
    \end{align*}
    Rearranging the above inequality yields the desired result.
\end{proof}
\section{Derivations of Convergence Rates} \label{app:proof-rates}
\subsection{Preparatory Tools}
We now state two results that have been shown in  \cite[Lemma 4 and 5]{Pol87}, which are crucial for convergence rate analysis. 
\begin{lemma} \label{lemma:rate} Let $\{y_k\}_{k}$ be a non-negative sequence and let $\varsigma \geq 0$, $d, p, q > 0$, $s \in (0,1)$, and $t > s$ be given constants. 
\begin{enumerate}[label=\textup{(\alph*)},topsep=0pt,itemsep=0ex,partopsep=0ex, leftmargin = 25pt]
		\item Suppose that the sequence $\{y_k\}_{k}$ satisfies 
		\[ 
		y_{k+1} \leq \Big(1- \frac{q}{k+\varsigma} \Big) y_k + \frac{d}{(k+\varsigma)^{p+1}}, \quad \forall \ k\geq 1.
		\]
		Then, if $q > p$, it holds that $ y_k \leq \frac{d}{q-p}\cdot (k+\varsigma)^{-p} + o((k+\varsigma)^{-p})$ (as $k \to \infty$). 
		\item Suppose that $\{y_k\}_{k}$ satisfies the recursion
		\[ y_{k+1} \leq \Big(1- \frac{q}{(k+\varsigma)^s} \Big) y_k + \frac{d}{(k+\varsigma)^{t}}, \quad \forall~k \geq 1. \]
		Then, it follows $y_k \leq \frac{d}{q} \cdot (k+\varsigma)^{s-t} + o((k+\varsigma)^{s-t})$ (as $k \to \infty$).
	\end{enumerate}
\end{lemma}
\begin{lemma}\label{lemma:step size}
	Let $\theta \in [0,1)$ be given and let $\{\alpha_k\}_{k}$ be defined via
	\[
	\alpha_k = \frac{\alpha}{(k+\varsigma)^\gamma}, \quad \alpha > 0, \quad \varsigma \geq 0, \quad \gamma \in (1/p,1]. 
	\]
\begin{enumerate}[label=\textup{(\alph*)},topsep=0pt,itemsep=0ex,partopsep=0ex, leftmargin = 25pt]
\item For all $k \geq 1$, we have 
\begin{equation}
	\label{eq:esti-u}
	{\sum}_{j=k}^\infty\alpha_j^p \leq  \frac{a_\gamma}{(k+\varsigma)^{p\gamma-1}} \quad \text{and} \quad \alpha_k \left[{\sum}_{j=k}^\infty\alpha_j^p\right]^{2\theta}\leq \frac{a_\gamma}{(k+\varsigma)^{(1+2p\theta)\gamma-2\theta}},
\end{equation}
where $a_\gamma>0$ is a numerical constant depending on $\alpha$ and $\gamma$.
	\item Moreover, if $\gamma >\frac{1+\theta}{1+p\theta}$, then it holds that 
	\begin{equation}\label{eq:esti-eps} 
	{\sum}_{t=k}^\infty \alpha_t \left[{\sum}_{j=t}^\infty\alpha_j^p\right]^{\theta} \leq \frac{{a}_\theta}{(k+\varsigma)^{(1+p\theta)\gamma-(1+\theta)}} \quad \forall~k \geq 1, 
	\end{equation}
	where ${a}_\theta>0$ is a numerical constant (depending on $\alpha, \gamma, \theta$).
\end{enumerate}	
\end{lemma} 
\begin{proof}
Using the integral test and noting that $p\gamma>1$, we obtain 
\[
{\sum}_{j=k}^\infty\alpha_j^p = {\sum}_{j=k}^\infty \;\frac{\alpha^p}{(j+\varsigma)^{p\gamma}} \leq \frac{\alpha^p}{(k+\varsigma)^{p\gamma}} + \int_{k}^\infty \frac{\alpha^p}{(x+\varsigma)^{p\gamma}}\; \rmn{d}x \leq \frac{p\gamma \alpha^p}{p\gamma-1} \cdot  \frac{1}{(k+\varsigma)^{p\gamma-1}}.
\]
In addition, it follows $
\alpha_k [{\sum}_{j=k}^\infty\alpha_j^p]^{2\theta} \leq \alpha^{1+2p\theta}(\frac{p\gamma}{p\gamma-1})^{2} \cdot \frac{1}{(k+\varsigma)^{2\theta(p\gamma-1)+\gamma}}$. Note that $\alpha^{1+2p\theta} \leq (1+\alpha)^{1+2p}$, which completes the proof of the statement (a). We refer to \cite[Lemma 3.7]{li2023convergence} for the proof of statement (b).  
\end{proof}
\begin{lemma}\label{lem:in-main-proof}
Let \cref{as:func-1} and \ref{R2} hold. Assume there exists $k, t \geq 1$ such that $\sum_{j=0}^{t-1} \alpha_{k+j} \leq 1/(3c\sL)$, then   %$\|f (x^{I_{k}})\|^2 \leq   2(c_2/a_2\varsigma)^2({\sum}_{j=0}^{i-1} q_{k+j})^2 + 8\|\nabla f (x^{I_{k+i}})\|^2$. 
\[\|\nabla f (x^{k+i})\|^2 \geq \frac18\|\nabla  f (x^{k})\|^2-  \frac{1}{4}\alpha_{k}^{2q-2},\quad \forall\, 0\leq i\leq t.\]
\end{lemma}
We introduce Gronwall's inequality (cf. \cite[Appendix B, Lemma 8]{borkar2009stochastic}) to facilitate the derivation of \cref{lem:in-main-proof}.
	\begin{lemma}[Gronwall's Inequality]
		\label{Thm:G}
		Let $\{a_k\}_k \subseteq \R_{++}$ and $\{y_k\}_k \subseteq \R_+$ be given sequences. Suppose that we have $y_{t+1} \leq s + r \sum_{j=0}^t a_j y_j$ for all $t$ and some $s,r\geq0$. 
		Then, it holds that $y_{t+1} \leq s \cdot \exp({r\sum_{j=0}^t a_j})$ for all $t\geq 0$.
	\end{lemma}
\emph{Proof of \cref{lem:in-main-proof}}.
When $i=0$, the bound in \cref{lem:in-main-proof} holds trivially. It holds for any integer $i \in [1,t]$ that
\begin{equation}
		\label{eq:lem-gron-1}
			\|\nabla f (x^{k+i}) - \nabla f (x^{k})\| \leq \sL \|x^{k+i} - x^{k}\|.
\end{equation}
According to the condition \ref{R2} and ${\sum}_{j=0}^{i-1} \;\alpha_{k+j} \leq 1/(3c\sL)$, we may bound the term $\|x^{k+i}-x^{k}\|$ as:	
	\begin{equation}
		\label{eq:lem-gron-2}
		\begin{aligned}
		\|x^{k+i}-x^{k}\| & \leq  c {\sum}_{j=0}^{i-1} \alpha_{k+j}\|\nabla f(x^{k+j})\| +  c {\sum}_{j=0}^{i-1} \alpha_{k+j}^q \\&\leq  (3\sL)^{-1}\|\nabla f(x^{k})\| + c \;{\sum}_{j=0}^{i-1} \alpha_{k+j}\|\nabla f(x^{k+j}) - \nabla f(x^k)\| +   c \;{\sum}_{j=0}^{i-1} \alpha_{k+j}^q.
		\end{aligned}
	\end{equation}		
Combining estimates \eqref{eq:lem-gron-1} and \eqref{eq:lem-gron-2}, then
\[
\|\nabla f (x^{k+i}) - \nabla f (x^{k})\| \leq c \sL \,{\sum}_{j=0}^{i-1} \,\alpha_{k+j}\|\nabla f(x^{k+j}) - \nabla f(x^{k})\|  + \tfrac13\|\nabla f(x^{k})\|  + c \sL \, {\sum}_{j=0}^{i-1} \alpha_{k+j}^q.
\]
Invoking Gronwall's inequality (\cref{Thm:G}) upon setting
		\begin{align*}
			s&:=\tfrac13 \|\nabla f(x^{k})\|  +  c \sL \, {\sum}_{j=0}^{i-1} \alpha_{k+j}^q, \quad 
			r:= c \sL,\quad a_j:= \alpha_{k+j}, \quad
			y_j:=\|\nabla f(x^{k+j}) - \nabla f(x^{k})\| %,\quad \text{and}\quad t:=i-1,
		\end{align*}
		and utilizing ${\sum}_{j=0}^{i-1} \;\alpha_{k+j} \leq 1/(3c\sL)$, we obtain \[\|\nabla f (x^{k+i}) - \nabla f (x^{k})\|  \leq \exp(1/3) \cdot \big(\tfrac13 \|\nabla f(x^{k})\|  +  c \sL \, {\sum}_{j=0}^{i-1} \alpha_{k+j}^q \big).\]
	Noticing that $\exp(1/3)\leq \frac32$, we then invoke the triangle inequality and the above estimate to yield	\[
		\|\nabla  f (x^{k})\| \leq \|\nabla f (x^{k+i}) - \nabla f (x^{k})\| + \|\nabla f (x^{k+i})\| \leq \frac12 \|f (x^{k})\| + \frac{3c\sL}{2}
		{\sum}_{j=0}^{i-1} \alpha_{k+j}^q + \|\nabla f (x^{k+i})\|.
	\]	 
	Rearranging the above estimate gives $\|f (x^{k})\| \leq 3c\sL{\sum}_{j=0}^{i-1} \alpha_{k+j}^q + 2\|\nabla f (x^{k+i})\|$. Taking square of both sides and using $(a+b)^2 \leq 2a^2 + 2b^2 $, we have $\|f (x^{k})\|^2 \leq   18c^2\sL^2({\sum}_{j=0}^{i-1} \alpha_{k+j}^q)^2 + 8\|\nabla f (x^{k+i})\|^2$. Since $\{\alpha_k\}_k$ is non-increasing, by using ${\sum}_{j=0}^{i-1} \;\alpha_{k+j} \leq 1/(3c\sL)$, we have
	\[
	\Big({\sum}_{j=0}^{i-1} \alpha_{k+j}^q\Big)^2 \leq  \Big(\alpha_k^{q-1} {\sum}_{j=0}^{i-1}  \alpha_{k+j}\Big)^2\leq \frac{1}{9c^2\sL^2} \cdot  \alpha_{k}^{2q-2}.
	\]
Merging this estimate into $\|f (x^{k})\|^2 \leq   18c^2\sL^2({\sum}_{j=0}^{i-1} \alpha_{k+j}^q)^2 + 8\|\nabla f (x^{k+i})\|^2$ completes the proof.\hfill \qedsymbol{}
\subsection{Iterate Convergence}\label{subsec:iter-conv}
\begin{proof}
    Conditions \ref{R1}--\ref{R3} verifies \ref{C1}--\ref{C3} with 
\[I_k = k,\quad \beta_k = \alpha_k,\quad p_k=\tilde b \alpha_k^p,\quad q_k=\tilde c \alpha_k^q, \quad \text{where} \quad p,q>1. \]
Next, we show that $\alpha_k = \alpha/ k^\gamma$ with $\gamma \in( \min\{\frac{2}{p+1},\frac{1}{q}\},1]$ satisfies \ref{C4}. To that end, we need to show 
\[
\sum_{k=1}^{\infty}\, \alpha_k = \infty,\quad \sum_{k=1}^{\infty}\,\alpha_k \left[{\sum}_{t=k}^{\infty}\alpha_t^p\right]^\mu<\infty,\quad \text{and}\quad \sum_{k=1}^{\infty}\,\alpha_k \left[{\sum}_{t=k}^{\infty} \alpha_t^{2q-1}\right]^\mu <\infty\quad \text{for some $\mu\in(0,1)$}.
\]
The condition $\sum_{k=1}^{\infty}\, \alpha_k = \infty$ holds for all $\gamma \in(0,1]$. Next, it follows from \cref{lemma:step size} (b) and $\gamma > \frac{2}{p+1}$ that $\sum_{k=1}^{\infty}\,\alpha_k \left[{\sum}_{t=k}^{\infty}\alpha_t^p\right]^\mu<\infty$ when $\mu>\frac{1-\gamma}{p\gamma-1} \in(0,1)$. Similarly, we can show $\sum_{k=1}^{\infty}\,\alpha_k [{\sum}_{t=k}^{\infty} \alpha_t^{2q-1}]^\mu <\infty$. 

Therefore, invoking \cref{thm:main} (b), we conclude that $\{x^k\}_k$ converges to some stationary point $x^*\in\crit(f)$.
\end{proof}

\subsection{Proof of \cref{thm:convergence rate}} \label{proof:convergence rate}
Clearly, the conditions \ref{R1}--\ref{R2} imply \ref{C1}--\ref{C2} with 
\[I_k = k,\quad \beta_k = \alpha_k,\quad p_k=\tilde b \alpha_k^p,\quad q_k=\tilde c \alpha_k^q, \quad \text{where} \quad p,q>1\quad\text{and} \quad p+1\leq 2q. \]
In addition, the sequence $\{\alpha_k\}_k$ satisfies \ref{C3}. Hence, by \cref{thm:main} (a), it follows
\begin{equation} {\lim}_{k\to\infty} \|\nabla f(x^k)\| = 0\quad \text{and}\quad {\lim}_{k\to\infty}f(x^k) =  f^*. \label{eq:app:hihi} \end{equation}
We first assume that $\{x^k\}_k$ is bounded (which implies boundedness of $\acc(\{x^k\}_k)$). Since every accumulation point of $\{x^k\}_k$ is a stationary point of $f$, \cref{eq:app:hihi} readily implies $\acc(\{x^k\}_k)\subseteq\crit(f)$. Moreover, we have $f(\bar x)=f^*$ for all $\bar x\in\acc(\{x^k\}_k)$ owing to the continuity of $f$. Hence, there exists an integer $k_1 \geq 1$ such that $x^k \in V_{\varepsilon,\eta}$ and $|f(x^k) - f^*|<\min\{1,\eta\}$ for all $k\geq k_1$, where $V_{\varepsilon,\eta}\subset\Rn$ is in defined in \cref{lem:uniform-kl}. By \cref{as:kl} and \cref{lem:uniform-kl}, the KL inequality 
$\|\nabla f(x^k)\| \geq {\sC} |f(x^k) - f^*|^\theta$ holds for all $k \geq k_1$, where $\sC>0$ and $\theta \in [\frac12,1)$ is the uniformized KL exponent of $f$ on $\acc(\{x^k\}_k)$. To obtain the optimal rates of convergence, we invoke the KL inequality with the adjusted KL exponent $\vartheta \in [\theta,1)$ in the subsequent analysis, i.e.,
\begin{equation}
	\label{eq:kl-adjusted}
	\|\nabla f(x^k) \| \geq \sC |f(x^k) - f^*|^\vartheta.
\end{equation}
Inequality \cref{eq:kl-adjusted} is true for all $k \geq k_1$ because $|f(x^k) - f^*|<1$ and $|f(x^k) - f^*|^\theta \geq |f(x^k) - f^*|^\vartheta$.

Now, we are in the position to derive the rates for $\{f(x^k)\}_k$ and $\{\|\nabla f(x^k)\|\}_k$. We first introduce an auxiliary sequence $\{\Gamma_k\}_k$ to facilitate the rate analysis: 
\begin{equation}
	\label{def:Gamma-u}
	\Gamma_k :=f(x^k) + u_k - f^* \quad \text{and} \quad u_k:=\tilde b\;{\sum}_{i=k}^\infty \alpha_i^p \leq \sG /k^{p\gamma-1},\quad \text{for some $\sG>0$},
\end{equation}
where the last inequality is due to \Cref{lemma:step size}. Rearranging the estimate in \ref{R1} gives
\begin{equation}
    \label{eq:descent-again}
     \Gamma_k - \Gamma_{k+1} \geq b\alpha_k\|\nabla f(x^k)\|^2 \geq \sC^2b\alpha_k|f(x^k) - f^*|^{2\vartheta},
\end{equation}
where the last inequality follows from \eqref{eq:kl-adjusted}. In addition, applying the inequality $2|a|^{2\vartheta} + 2|b|^{2\vartheta} \geq |a + b|^{2\vartheta}$, $\vartheta \in[\frac12,1)$, $a,b \in \R$, we obtain $|f(x^k) - f^*|^{2\vartheta} + u_k^{2\vartheta} \geq \Gamma_k^{2\vartheta}/2$, which, along with the estimate \eqref{eq:descent-again}, leads to
	\[
	\Gamma_k - \Gamma_{k+1}  \geq \frac{\sC^2b\alpha_k}{2}(\Gamma_k^{2\vartheta} - 2 u_k^{2\vartheta}).
	\]
Substituting $\alpha_k = \alpha/k^\gamma$, invoking \cref{def:Gamma-u}, and rearranging this inequality, it holds that 
	\begin{equation}
	 	\label{eq:recursion-Gamma}
	 	\Gamma_{k+1} \leq \Gamma_k - \frac{\sC^2 b\alpha}{2} \cdot \frac{\Gamma_k^{2\vartheta}}{k^\gamma} +\frac{\sH}{k^{(1+2p\vartheta)\gamma-2\vartheta}},\quad\text{where}\quad\sH:=\sC^2\sG^{2\vartheta}b\alpha.
	 \end{equation}
	 In what follows, we will prove the convergence rate of $\{\Gamma_k\}_k$ given different $\vartheta\in[\frac12,1)$ and $\gamma>0$. 

\noindent\textbf{Step 1: Rates for auxiliary sequence $\Gamma_k$}

	\underline{Case 1:} $\vartheta =\frac12$. In this case, the estimate \eqref{eq:recursion-Gamma} reduces to
\[		\Gamma_{k+1} \leq \left[1-\frac{\sC^2b \alpha}{2}\cdot \frac{1}{k^\gamma} \right]\Gamma_{k} + \frac{\sH}{k^{(1+p)\gamma-1}}.
\]
	When $\gamma < 1$, we have $\Gamma_k = \cO(1/k^{p\gamma-1})$ by \Cref{lemma:rate} (b). When $\gamma=1$ and $\alpha > \frac{2(p-1)}{b\sC^2}$, then \Cref{lemma:rate} (a) yields $\Gamma_k = \mathcal O(1/k^{p-1})$.

	\underline{Case 2:} $\vartheta \in (\frac12,1)$, $\gamma \neq 1$. Similar to Case 1, we adopt a similar approach to derive the rate for the sequence $\{\Gamma_k\}_k$. Our strategy involves transforming the recursion \eqref{eq:recursion-Gamma} into a suitable form that allows us to apply \Cref{lemma:rate}.
	
	To this end, we first define a mapping $x \mapsto h_\vartheta(x):= x^{2\vartheta}$, which is convex for all $x>0$, i.e.,
	\begin{equation}
		\label{eq:convexity-h}
		h_\vartheta(y) \geq  h_\vartheta(x) + h_\vartheta^\prime(x)(y-x) = 2\vartheta x^{2\vartheta-1} y +  (1 - 2\vartheta )h_\vartheta(x) \geq x^{2\vartheta-1} y - h_\vartheta(x) \quad \forall~x,y>0,\;\vartheta\in(\tfrac12,1).
	\end{equation}
	Clearly, it holds that $h_\vartheta(\Gamma_k) = \Gamma_k^{2\vartheta}$. We now leverage the inequality \eqref{eq:convexity-h} to obtain a lower bound for $\Gamma_k^{2\vartheta}$. Specifically, we apply \eqref{eq:convexity-h} with
\[
	x:=\tilde\sC/k^{\delta} \quad \text{and} \quad y:=\Gamma_k,\quad \text{where}\quad {\textstyle \delta:=\min\left\{\frac{1-\gamma}{2\vartheta-1}, p\gamma-1 \right\} \quad \text{and} \quad\tilde\sC= \left(\frac{4\delta}{b\sC^2 \alpha}\right)^{\frac{1}{2\vartheta-1}}},
	\]
	 which yields $\Gamma_k^{2\vartheta} \geq \frac{4\delta}{b\sC^2\alpha} \cdot  (\Gamma_k/k^{(2\vartheta-1)\delta}) - \tilde\sC^{2\vartheta}/k^{2\vartheta\delta} $. Combining this inequality with \eqref{eq:recursion-Gamma}, it holds that
	 \[
	 \Gamma_{k+1} \leq \left[1- \frac{2\delta}{k^{\gamma+(2\vartheta-1)\delta}} \right]\Gamma_k + \frac{\sC^2 \tilde \sC^{2\vartheta}b\alpha}{2}\cdot\frac{1}{k^{\gamma+2\vartheta\delta}}  +\frac{\sH}{k^{(1+2p\vartheta)\gamma-2\vartheta}}.
	 \]
	By the definition of $\delta$, it follows that $\gamma+2\vartheta \delta \leq (1+2p\vartheta)\gamma-2\vartheta$. Hence, $\cO(1/k^{\gamma+2\vartheta\delta})$ is the leading term and there exists a constant $\hat \sC>0$ such that 
	\[
	 \Gamma_{k+1} \leq \left[1- \frac{2\delta}{k^{\gamma+(2\vartheta-1)\delta}} \right]\Gamma_k + \frac{\hat \sC}{k^{\gamma+2\vartheta\delta}}.
	 \]
	Applying \Cref{lemma:rate} yields $\Gamma_k=\cO(k^{-\delta})$.  Our aim is to obtain the optimal rate for the given step sizes parameter $\gamma>0$. Note that the rate parameter $\delta>0$ is determined by the adjusted KL exponent $\vartheta$ which can be chosen freely in the interval $[\theta , 1)$. Moreover, $\delta$ is a non-increasing function of $\vartheta$, and thus, the optimal choice is fixing $\vartheta = \theta$. On the other hand, if $\theta=\frac12$, we simply set $\vartheta=\frac12$ and the rate is provided in Case 1 .
	
	To summarize, we provide the rate of $\{\Gamma_k\}_k$ in terms of the original KL exponent $\theta$ and the step sizes parameter $\gamma$ :
	\begin{equation}\label{eq:rate-Gamma}
		\begin{aligned}
		\Gamma_k = \cO(k^{-\psi(\theta,\gamma)})\quad \text{where}\quad \psi(\theta,\gamma):=\min\{p\gamma-1, \tfrac{1-\gamma}{2\theta-1}\}\quad\text{for} \quad \gamma \in (\textstyle\frac1p,1),
	\end{aligned}
	\end{equation}
and $\psi(\tfrac12,1):=p-1$ when $\alpha > \frac{2(p-1)}{b\sC^2}$.

\textbf{Step 2: Rates for $f(x^k)$}

Recall the definition of $\Gamma_k =f(x^k) + u_k - f^*$. Using triangle inequality and the fact that $\Gamma_k,u_k \geq 0$, we obtain
\[
|f(x^k) - f^*| = |f(x^k) + u_k - f^* - u_k | \leq \Gamma_k + u_k.
\]
From \eqref{def:Gamma-u}, it holds that $u_k = \cO(1/k^{p\gamma-1})$. Utilizing the rate $\Gamma_k = \cO(k^{-\psi(\theta,\gamma)})$ in \eqref{eq:rate-Gamma} and noting that $\psi(\theta,\gamma) \leq p\gamma-1$ for all $\theta\in[0,1)$ and $\gamma \in (\frac1p,1]$, we thus obtain  
\begin{equation}
	\label{eq:rate-func-z}
|f(x^k) - f^*| = \cO(k^{-\psi(\theta,\gamma)}).
\end{equation}

\textbf{Step 3: Rates for $\|\nabla f(x^k)\|^2$}

In this step, we will derive the rate for $\|\nabla f(x^k)\|^2$ using the rate of $\{\Gamma_k\}_k$ in \eqref{eq:rate-Gamma}. We now begin with some discussion to motivate why a more dedicated technique is needed to achieve better rates.

\textbf{Step 3-1: Roadmap of the proof}

Clearly, the relation \eqref{eq:descent-again} leads to $b\alpha_k\|\nabla f(x^k)\|^2 \leq \Gamma_k$, and it follows directly from the rates for $\{\Gamma_k\}_k$ in \eqref{eq:rate-Gamma} that $\|\nabla f(x^k)\|^2=\cO(k^{-\psi(\theta,\gamma)+\gamma})$. But we will soon see in the subsequent analysis that this rate can be improved. 

Let $\varpi:\N\to\N$ denote an index mapping and let us sum \eqref{eq:descent-again} from $k$ to $\varpi(k)$: %
\begin{equation}
	\label{eq:rate-grad-1}
b{\sum}_{i=k}^{\varpi(k)}\alpha_i\|\nabla f(x^i)\|^2 \leq \Gamma_k.
\end{equation}
We will provide a detailed definition of $\varpi$ later. Our next step is to connect $\|\nabla f(x^i)\|$ and $\|\nabla f(x^k)\|$. We use a simple example to emphasize the importance of bounding $\|\nabla f(x^k)\|$ by $\|\nabla f(x^i)\|$ in obtaining the improved rate. Assume $\|\nabla f(x^i)\| \geq \|\nabla f(x^k)\|$ for all integers $i \in [k,\varpi(k)]$ and assume $\sum_{i=k}^{\varpi(k)}\alpha_i \geq \bar \alpha$ for some constant $\bar \alpha>0$ and for all $k$. Then, we can further write \eqref{eq:rate-grad-1} as
\[
b\bar \alpha\|\nabla f(x^k)\|^2 \leq b{\sum}_{i=k}^{\varpi(k)}\alpha_i\|\nabla f(x^i)\|^2 \leq \Gamma_k.
\]
In this way, we obtain $\|\nabla f(x^k)\|^2=\cO(k^{-\psi(\theta,\gamma)})$, which is faster than $\cO(k^{-\psi(\theta,\gamma)+\gamma})$. In what follows, we will show the existence of the mapping $\varpi$ such that the quantity $\sum_{i=k}^{\varpi(k)}\alpha_i$ fulfills certain bounds for all $k$ sufficiently large ensuring that \cref{lem:in-main-proof} is applicable. (\cref{lem:in-main-proof} will provide the required upper bound of $\|\nabla f(x^k)\|$).

\textbf{Step 3-2: The mapping $\varpi:\N\to\N$}

Let us define the mapping \[\varpi(k):=\sup_{t \geq 1}\bigg\{ {\sum}_{j=0}^{t-1}\alpha_{k+j} \leq \frac{1}{3c\sL}\bigg\}.\] Since $\alpha_k \to 0$ and $\sum_{k=1}^\infty \alpha_k = \infty$, we conclude that $\varpi(k)$ is finite and non-empty for all $k$ sufficiently large. Then, there is $k_2 \geq k_1$ such that $\alpha_k \leq \frac{1}{6c\sL}$ and $\varpi(k)$ is well-defined for all $k \geq k_2$. By the definition of $\varpi$, we have ${\sum}_{j=0}^{\varpi(k)-1}\alpha_{k+j} \leq \frac{1}{3c\sL}$ and
	\begin{equation}\label{eq:rate-grad-4}
\frac{1}{3c\sL} < %{\sum}_{j=0}^{\varpi(k)}\alpha_{k+j} = 
\alpha_{k+\varpi(k)} + {\sum}_{j=0}^{\varpi(k)-1}\alpha_{k+j} \leq \frac{1}{6c\sL}+{\sum}_{j=0}^{\varpi(k)-1}\alpha_{k+j}  \quad \Longrightarrow \quad \frac{1}{6c\sL} \leq {\sum}_{j=0}^{\varpi(k)-1}\alpha_{k+j} \leq \frac{1}{3c\sL}.
\end{equation}
Hence, \cref{lem:in-main-proof} is applicable for all $k \geq k_2$. 

\textbf{Step 3-3: Rates for $\|\nabla f(x^k)\|$}

Applying \cref{lem:in-main-proof}  with $t=\varpi(k)$ and utilizing the estimate \eqref{eq:rate-grad-1}, this yields
\begin{align*}
	\Gamma_k &\geq b{\sum}_{j=0}^{t-1}\alpha_{k+j}\|\nabla f(x^{k+j})\|^2 \geq \frac{b}{8} \Big({\sum}_{j=0}^{t-1}\alpha_{k+j}\Big)\|\nabla f(x^k)\|^2 - \frac{b}{4}\alpha_k^{2q-2} \cdot \Big({\sum}_{j=0}^{t-1}\alpha_{k+j}\Big) \\ &\geq \frac{b}{48c\sL}\|\nabla f(x^k)\|^2 - \frac{b}{12c\sL} \cdot \alpha_k^{2q-2},
\end{align*}	
where the last inequality is due to the bound \eqref{eq:rate-grad-4}. Rearranging this inequality, substituting $\alpha_k = \frac{\alpha}{k^\gamma}$, and using the rates for $\{\Gamma_k\}_k$ in \eqref{eq:rate-Gamma}, we have for all $k$ sufficiently large,
\begin{equation}
	\label{eq:rate-grad-x}
 \|\nabla f(x^k) \|^2 = \cO(k^{-\psi(\theta,\gamma)} + k^{-(2q-2)\gamma}) = \cO(k^{-\psi(\theta,\gamma)}),
\end{equation}
where the last relation holds because $(2q-2)\gamma \geq (p-1)\gamma \geq \psi(\theta,\gamma)$.

\textbf{Step 4: Discussion}

So far, we have established the convergence rates for $\{f(x^k)\}_k$ and $\{\|\nabla f(x^k)\|^2\}_k$ shown in \cref{thm:convergence rate} for all $\gamma \in (\frac{1}{p},1]$, provided that $\{x^k\}_k$ is bounded, and where $\theta$ denotes the uniformized KL-exponent of $f$ on $\acc(\{x^k\}_k)$. 

Alternatively, suppose now that $\acc(\{x^k\}_k)$ is non-empty and $\gamma$ is chosen via $\gamma \in (\frac{2}{p+1},1]$. In this scenario, as argued in \cref{subsec:iter-conv}, we have $\sum_{k=1}^\infty \alpha_k \left[{\sum}_{t=k}^\infty \alpha_k^{p}\right]^\mu < \infty$ for $\mu > \frac{1-\gamma}{p\gamma-1} \in (0,1)$ and hence, \cref{thm:main} (b) is applicable and $\{x^k\}_k$ converges to some stationary point $x^*\in\crit(f)$. Consequently, in this case, all the previous estimates and derivations hold with $\theta$ being the KL-exponent of $f$ at $x^*$. The proof of \cref{thm:convergence rate} is complete if we can establish the stated convergence rates for $\{x^k\}_k$. We handle this task in the next (and last) step.
 
\textbf{Step 5: Rates for $\{x^k\}_k$}

As before, due to $x^k \to x^*$ and $f(x^k) \to f(x^*) = f^*$, there exists an integer $k_3 \geq 1$ such that we have
\[
\|\nabla f(x^k)\| \geq {\sC} |f(x^k) - f^*|^\theta \geq \sC|f(x^k) - f^*|^\vartheta, \quad \vartheta\in[\theta,1]
\]
for all $k \geq k_1$. (Here, $\theta$ is the KL-exponent of $f$ at $x^*$).  Moreover, \cref{lem:dk-iter} is applicable for all $j \geq k_3$, i.e., it holds that
  \begin{equation}\label{eq:rate-iter-1}
		\|x^j - x^{j+1}\|  \leq  \frac{\sqrt2 c}{b}(\Delta_j-\Delta_{j+1}) + \sqrt2 c \sC \alpha_j v_j^{\vartheta},
    \end{equation}
    where $\Delta_j:=\frac{1}{\sC(1-\vartheta)}[f(x^{j})- f^* + v_j]^{1-\vartheta}$ and $v_j:= \tilde b \sum_{i=j}^{\infty}\alpha_i^p + \frac{b}{2c}\sum_{i=k}^{\infty}\alpha_i^{2q-1}$. 
Using the convergence of $\{x^k\}_k$ and the triangle inequality and summing \eqref{eq:rate-iter-1} from $j=k$ to infinity, we obtain
\[\|x^k - x^*\| \leq \sum_{j=k}^\infty\|x^j-x^{j+1}\| \leq \frac{\sqrt2 c}{b}\Delta_k + \sqrt2 c \sC \sum_{j=k}^\infty \alpha_j v_j^{\vartheta}.\]
Invoking $\alpha_k = \alpha/k^\gamma$ and $p \leq 2q-1$ and applying \Cref{lemma:step size}, we get the following rates:
\begin{equation}
    \label{eq:rate-v}
    v_k \leq b_\gamma/k^{p\gamma-1}\quad \text{and} \quad {\sum}_{j=k}^\infty \alpha_j v_j^{\vartheta} \leq c_\gamma / k^{(p\gamma-1)\vartheta - (1-\gamma)}\quad \text{for some $b_\gamma,c_\gamma>0$.}
\end{equation}
Utilizing the rate of  $\{f(x^k)\}_k$ derived in \textbf{Step 2} and the rate of $\{v_k\}_k$, it follows
\[
\Delta_k \leq \frac{|f(x^k) - f^*|^{1-\vartheta}}{\sC(1-\vartheta)} +  \frac{b_\gamma^{1-\vartheta}}{\sC(1-\vartheta)} \cdot \frac{1}{k^{(p\gamma-1)(1-\vartheta)}} = \cO\Big(\frac{1}{k^{(1-\vartheta)\psi(\vartheta,\gamma)}}\Big).
\]
Combining this bound with \eqref{eq:rate-v}, we obtain the rate for $\{x^k\}_k$ in terms of the adjusted KL exponent $\vartheta$ (with $\frac{1-\gamma}{p\gamma-1} < \vartheta \in[\theta,1)$):
\[
\|x^k - x^*\| = \cO(k^{-\phi(\vartheta,\gamma)}),\quad\text{where} \quad  \phi(\vartheta,\gamma):=\min\{(1-\vartheta)\psi(\vartheta,\gamma), (p\gamma-1)\vartheta-(1-\gamma)\}.
\]
According to the definition of mapping $\psi$ in \eqref{eq:rate-Gamma}, it holds that 
\begin{align*}
	 \phi(\vartheta,\gamma):=\begin{cases}
		(p\gamma-1)\vartheta - (1-\gamma) &\text{if}\; \vartheta \in \big(\frac{1-\gamma}{p\gamma-1},\frac{(p-1)\gamma}{2(p\gamma-1)} \big]\\[2mm]
		\frac{(1-\vartheta)(1-\gamma)}{2\vartheta-1} &\text{if}\; \vartheta \in \big(\frac{(p-1)\gamma}{2(p\gamma-1)},1 \big)
	\end{cases}\quad \text{and} \quad \gamma \in (\tfrac{2}{p+1},1),
\end{align*}
 and $\phi(\vartheta,\gamma):=\frac{p-1}{2}$ when $\gamma=1, \vartheta=\frac12$ and $\alpha > \frac{2(p-1)}{b\sC^2}$.
 Note that the mapping $\phi(\cdot,\gamma)$ is increasing over the interval $(\frac{1-\gamma}{p\gamma-1},\frac{(p-1)\gamma}{2(p\gamma-1)}]$ and decreasing over $(\frac{(p-1)\gamma}{2(p\gamma-1)},1)$. Hence, the optimal choice of $\vartheta\in[\theta,1)$ that can maximize $\phi(\vartheta,\gamma)$ is to set $\vartheta=\frac{(p-1)\gamma}{2(p\gamma-1)}$ if $\theta\leq \frac{(p-1)\gamma}{2(p\gamma-1)}$, and $\vartheta=\theta$, otherwise. In conclusion, we have $\|x^k-x^*\| = \cO(k^{-\varphi(\theta,\gamma)})$, where  $\varphi(\theta,\gamma):=\max_{\theta \leq \vartheta <1} \phi(\vartheta,\gamma)$, that is, 
\begin{align*}
    \varphi(\theta,\gamma):=\begin{cases}
		\tfrac{(p+1)\gamma}{2}-1 &\text{if}\; \theta \in \big[\frac12,\frac{(p-1)\gamma}{2(p\gamma-1)}\big]\\[2mm]
		\frac{(1-\gamma)(1-\theta)}{2\theta-1} &\text{if}\; \theta \in \big(\frac{(p-1)\gamma}{2(p\gamma-1)},1\big)
	\end{cases}\quad \text{and} \quad \gamma \in (\tfrac{2}{p+2},1),
\end{align*}
and $\varphi(\theta,\gamma):=\frac{p-1}{2}$ when $\gamma=1, \theta=\frac12$ and $\alpha > \frac{2(p-1)}{b\sC^2}$. Alternatively, when $\gamma \in (\tfrac{2}{p+2},1)$, we can express the rates function as $\varphi(\theta,\gamma)=\min\{\tfrac{(p+1)\gamma}{2}-1,\frac{(1-\gamma)(1-\theta)}{2\theta-1}\}$.

\section{Appendix: Decentralized Gradient Descent}
For a more compact representation, we denote throughout this section 
\begin{align*}
	\bx:=\left[
	x_1,x_2,\cdots ,x_n
	\right]\in\R^{d\times n} \quad \text{and} \quad 
	%\\	F(\bx)&:=\left[f_1(x_1), f_2(x_2), \cdots, f_N(x_N)\right]\in\R^{1\times N},\\ 
	\nabla F(\bx):=\left[
	\nabla f_1(x_1),  \nabla f_2(x_2), \cdots,  \nabla f_n(x_n)
	\right] \in \R^{d\times n}.
\end{align*}
Then, the update \eqref{algo:dgd} of $\DGD$ can be written as
\[\bx^{k+1} = \bx^k W - \alpha_k \nabla F(\bx^k).\]  
	To show \cref{lemma:dgd} and \ref{prop:dgd}, we need the following supporting lemmas whose proof is presented in \cref{proof:lem:sublinear}.
\begin{lemma}
	\label{lem:sublinear}
	Assume the sequence $\{\alpha_k\} \subseteq \R_{++}$ satisfies $\lim_{k\to\infty} a_k =0$ and $\lim_{k\to\infty} \frac{a_{k+1}}{a_k}=1$. Then, for an arbitrary $\rho\in[0,1)$, there exists a constant $\tilde c>0$ such that 
	$\sum_{i=1}^{k} a_i \rho^{k-i} \leq \tilde c a_{k+1}$ for all $k \in \N$.
\end{lemma}
\begin{lemma}
	\label{lem:dgd-descent}
	Suppose \cref{as:func} and \cref{as:matrix}. Let $\{\bx^k\}_k$ be generated by \cref{algo:dgd} with the step size $\alpha_k\in\left(0,\frac{1}{{\sf L}n}\right)$, then $	f(\bar x^{k+1})  
	 \leq f(\bar x^k) - \frac{\alpha_k}{2}\|\nabla f(\bar x^k)\|^2  + \frac{\alpha_k\sL^2}{2n}\|\bar x^k \one^T -  \bx^k\|^2_F.$
	% where $\{\bar x^k\}_k:=\{\bx^k\one/N\}_k$.
\end{lemma}
\subsection{Proof of \cref{lemma:dgd}}\label{proof:lemma:dgd}
\begin{proof}
	By \cref{as:matrix} (b) and Cauchy-Schwartz inequality, 
    we have
	\begin{align*}
		\|\bx^{k+1} - \bar x^{k+1} \one^T\|_F &= \|\bx^{k+1}(I - \tfrac{1}{n}\one \one^T)\|_F \leq  \|\bx^1\|_F\|W^{k} - \tfrac{1}{n}\one \one^T\| +  \sum_{i=1}^{k}\alpha_i \|W^{k-i} - \tfrac{1}{n}\one \one^T\| \|\nabla F(\bx^i)\|_F\\
		&\leq \rho^{k}\|\bx^1\|_F + \sum_{i=1}^{k}\alpha_i \rho^{k-i} \|\nabla F(\bx^i)\|_F
	\end{align*}
where $\rho$ represents the second largest magnitude eigenvalue of $W$. Taking square on the both sides of the latter estimate yields
\[
		\|\bx^{k+1} - \bar x^{k+1} \one^T\|_F^2 \leq 2\rho^{2k}\|\bx^1\|_F^2 + 2\Big({\sum}_{i=1}^{k}\alpha_i \rho^{k-i} \|\nabla F(\bx^i)\|_F\Big)^2.
\]
Let us expand the last term and use the inequality $2ab\leq a^2+b^2$, then
\begin{align*}
	\Big( {\sum}_{i=1}^{k}\alpha_i \rho^{k-i} \|\nabla F(\bx^i)\|_F\Big)^2 & \leq \sum_{i=1}^{k}\alpha_i^2 \rho^{2k-2i} \|\nabla F(\bx^i)\|_F^2 + \sum_{i=2}^{k}\sum_{j=1}^{i-1}\alpha_i \alpha_j \rho^{2k-i-j}(\|\nabla F(\bx^i)\|_F^2+\|\nabla F(\bx^j)\|_F^2).
\end{align*}
By \cref{lem:sublinear}, there exists a constant $c_\rho>\frac{1}{1-\rho}$ such that,
\begin{align*}
	&\quad\sum_{i=2}^{k}\sum_{j=1}^{i-1}\alpha_i \alpha_j \rho^{2k-i-j}(\|\nabla F(\bx^i)\|_F^2+\|\nabla F(\bx^j)\|_F^2)\\
	 &= \sum_{i=2}^{k} \rho^{k-i} \alpha_i \|\nabla F(\bx^i)\|_F^2\sum_{j=1}^{i-1} \rho^{k-j}\alpha_j + \sum_{i=2}^{k} \rho^{k-i} \alpha_i \sum_{j=1}^{i-1} \rho^{k-j}\alpha_j \|\nabla F(\bx^j)\|_F^2\\
	 &\leq c_\rho \sum_{i=2}^{k} \rho^{k-i} \alpha_i^2 \|\nabla F(\bx^i)\|_F^2 + \sum_{j=1}^{k-1} \rho^{k-j} \alpha_j \|\nabla F(\bx^j)\|_F^2 \sum_{i=j+1}^{k} \rho^{k-i}  \alpha_i \leq 2c_\rho \sum_{i=1}^{k} \rho^{k-i} \alpha_i^2 \|\nabla F(\bx^i)\|_F^2,
\end{align*}
where in the last inequality, we use the fact that $\{\alpha_k\}_k$ is non-increasing and $\sum_{i=1}^{k} \rho^{k-i}<1/(1-\rho)$. 
As a result, $( \sum_{i=1}^{k}\alpha_i \rho^{k-i} \|\nabla F(\bx^i)\|_F)^2 \leq (1+2c_\rho)\sum_{i=1}^{k} \rho^{k-i} \alpha_i^2 \|\nabla F(\bx^i)\|_F^2$, and thus,
\[\|\bx^{k+1} - \bar x^{k+1} \one^T\|_F^2 \leq 2\rho^{2k}\|\bx^1\|_F^2 + 2(1+2c_\rho)\sum_{i=1}^{k} \rho^{k-i} \alpha_i^2 \|\nabla F(\bx^i)\|_F^2. \]
Invoking the Cauchy-Schwartz inequality and Lipschitz smoothness of each component function, we have 
\begin{align*}
	\|\nabla F(\bx^i)\|_F^2 \leq 2\|\nabla F(\bx^i) - \nabla F(\bar{x}^i\one^T)\|_F^2 + 2\|\nabla F(\bar{x}^i\one^T)\|_F^2 \leq 2{\sL}^2\|\bx^i-\bar{x}^i\one^T\|_F^2 + 4n\sL(f(\bar{x}^i) -\bar f).
\end{align*}
Consequently, by merging the last two estimates, we obtain
\begin{equation}
	\label{eq:dgd-est-1}
	\|\bx^{k+1} - \bar x^{k+1} \one^T\|_F^2 \leq c_1 \rho^{2k} + c_2 \sum_{i=1}^{k} \rho^{k-i} \alpha_i^2\|\bx^i-\bar{x}^i\one^T\|_F^2 + c_3\sum_{i=1}^{k} \rho^{k-i} \alpha_i^2(f(\bar{x}^i) -\bar f),
\end{equation}
where we define $c_1:= 2\|\bx^1\|_F^2, c_2:=4{\sL}^2(1+2c_\rho),$ and $c_3:=8n\sL (1+2c_\rho)$ to simplify the notations. Let us consider the summation
\begin{equation}
	\label{eq:sum-1}
	\begin{aligned}
		&\quad \sum_{k=1}^{t+1} \alpha_k\|\bx^k-\bar{x}^k\one^T\|_F^2 = \alpha_1\|\bx^1-\bar{x}^1\one^T\|_F^2 + \sum_{k=1}^{t} \alpha_{k+1}\|\bx^{k+1}-\bar{x}^{k+1}\one^T\|_F^2 \\
		&\leq  c_1 \sum_{k=0}^{t} \rho^{2k} \alpha_{k+1} + c_2 \sum_{k=1}^{t}\alpha_{k+1} \sum_{i=1}^{k}\rho^{k-i}\alpha_i^2\|\bx^i-\bar{x}^i\one^T\|_F^2 +  c_3 \sum_{k=1}^{t}\alpha_{k+1} \sum_{i=1}^{k}\rho^{k-i}\alpha_i^2(f(\bar{x}^i) -\bar f)\\
		&\leq \frac{c_1\alpha_1}{1-\rho^2} + c_2 \sum_{i=1}^{t}\alpha_i^2\|\bx^i-\bar{x}^i\one^T\|_F^2 \sum_{k=i}^{t}\rho^{k-i}\alpha_{k+1} + c_3 \sum_{i=1}^{t}\alpha_i^2(f(\bar{x}^i) -\bar f) \sum_{k=i}^{t}\rho^{k-i}\alpha_{k+1}\\
		&\leq \frac{c_1\alpha_1}{1-\rho^2} + \frac{c_2}{1-\rho}\sum_{k=1}^{t+1}\alpha_k^3\|\bx^k-\bar{x}^k\one^T\|_F^2 + \frac{c_3}{1-\rho}\sum_{k=1}^{t+1}\alpha_k^3(f(\bar{x}^i) -\bar f),
	\end{aligned}
\end{equation}

where both the third and fourth inequalities follow from the nonincreasing property of $\{\alpha_k\}_k$ and the inequality $\sum_{i=0}^{k}\rho^i<(1-\rho)^{-1}$ for all $k\in\N$. Notice that $\alpha_k\to0$ as $k$ tends to infinity, there exists $k_0\in \N$ such that $\frac{c_2\alpha_k^2}{1-\rho}\leq \frac12$ for all $k\geq k_0$. Hence, by rearranging the terms in \eqref{eq:sum-1}, we obtain for all $t\geq k_0$
\begin{equation}
	\label{eq:dgd-est-2}
	\sum_{k=1}^{t+1} \alpha_k\|\bx^k-\bar{x}^k\one^T\|_F^2 \leq \frac{2c_3}{1-\rho}\sum_{k=1}^{t+1}\alpha_k^3(f(\bar{x}^i) -\bar f) + \underbracket[.75pt][4pt]{\frac{2c_1\alpha_1}{1-\rho^2} + \sum_{k=1}^{k_0}\Big(\frac{2c_2\alpha_k^2}{1-\rho}-1\Big)\alpha_k\|\bx^k-\bar{x}^k\one^T\|_F^2}_{=:c_4}.
\end{equation}
Equipped with \eqref{eq:dgd-est-2}, we are ready to show that $\{f(\bar x^k)\}_k$ has a uniform upper bound for all $k$. By \cref{lem:dgd-descent}, we have\[f(\bar x^{k+1}) - \bar f \leq f(\bar x^{k})  - \bar f + \frac{{\sL^2}\alpha_k}{2n} \|\bx^{k} - \bar x^{k} \one^T\|_F^2.\]
Unfolding the above recursion by summing over $k=1,2,\cdots, t$ and applying \cref{eq:dgd-est-2},
\[f(\bar x^{t+1}) - \bar f \leq f(\bar x^{1})  - \bar f + \frac{{\sL^2}}{2n}\sum_{k=1}^{t} \alpha_k\|\bx^k-\bar{x}^k\one^T\|_F^2\leq f(\bar x^{1})  - \bar f + \frac{c_4{\sL^2}}{2n} +  \frac{c_3{\sL^2}}{n(1-\rho)}\sum_{k=1}^{t}\alpha_k^3(f(\bar{x}^i) -\bar f).\]
Define the sequences $\{\phi_k\}_k$ and $\{\gamma_k\}_k$ as:
\[\phi_1:= f(\bar x^{1})  - \bar f +  \frac{c_4{\sL^2}}{2n},\quad \phi_k:=f(\bar x^{k})  - \bar f,\quad \gamma_1:=1+\frac{c_3{\sL^2}\alpha_1^3}{n(1-\rho)},\quad \gamma_k:=\frac{c_3{\sL^2}\alpha_k^3}{n(1-\rho)},\quad \forall k\geq 2.\]
Hence, the recursion is rewritten as
$\phi_{t+1} \leq \sum_{k=1}^{t} \gamma_k \phi_k$, where $\sum_{k=1}^{\infty}\gamma_k<\infty.$
Let us notice
\begin{align*}
	\sum_{t=1}^{\infty} \gamma_{t+1}\phi_{t+1} \leq \sum_{t=1}^{\infty}\gamma_{t+1} \sum_{k=1}^{t}\gamma_k \phi_k = \sum_{k=1}^{\infty} \gamma_k \phi_k \sum_{t=k}^{\infty}\gamma_{t+1}.
\end{align*}
There exists $k_1\in\N$ such that $\sum_{t=k}^{\infty}\gamma_{t+1}\leq \frac12$ for all $k > k_1$ because the sequence $\{\gamma_k\}_k$ is summable. Consequently,
\[\sum_{k=1}^{\infty}\gamma_k \phi_k \leq \gamma_1\phi_1 + c_5 \sum_{k=1}^{k_1}\gamma_k \phi_k + \frac12 \sum_{k=k_1+1}^{\infty}\gamma_k \phi_k \quad \Longleftrightarrow\quad \sum_{k=1}^{\infty}\gamma_k \phi_k \leq 2\gamma_1\phi_1 +(2c_5-1) \sum_{k=1}^{k_1}\gamma_k \phi_k<\infty,\]
where $c_5:=\sum_{k=1}^{\infty}\gamma_k$. Recalling $\phi_{t+1} \leq \sum_{k=0}^{t} \gamma_k \phi_k<\infty$ for all $t\in\N$, we conclude that the sequence $\{f(\bar{x}^k)\}_k$ is bounded from above. Let $\bar{\phi}$ denote the upper bound of the sequence $\{f(\bar{x}^k) - \bar f\}_k$, we have for all $k \geq 1$,
\begin{align*}
		&\quad \|\bx^{k+1} - \bar x^{k+1} \one^T\|_F^2 \leq c_1 \rho^{2k} + c_2 \sum_{i=1}^{k} \rho^{k-i} \alpha_i^2\|\bx^i-\bar{x}^i\one^T\|_F^2 + c_3\bar{\phi}\sum_{i=1}^{k} \rho^{k-i} \alpha_i^2\\
		&\leq c_1 \rho^{2k} + \frac{c_3\bar{\phi}\alpha_1^2}{1-\rho} + \frac{\alpha_1 c_2}{1-\rho} \sum_{i=1}^{k} \alpha_i\|\bx^i-\bar{x}^i\one^T\|_F^2 \leq c_1 \rho^{2k} + \frac{c_3\bar{\phi}\alpha_1^2}{1-\rho} + \frac{\alpha_1 c_2 c_4}{1-\rho} + \frac{2n\alpha_1\bar{\phi}c_2c_5}{(1-\rho){\sL^2}} =:\hat \sG,
\end{align*}
where the first line is due to the estimate \eqref{eq:dgd-est-1} and the last line follows from \eqref{eq:dgd-est-2}. Denote $\bar \sG :=\max\{\hat \sG,\bar \phi\}$. Based on the estimate \eqref{eq:dgd-est-1} and \cref{lem:sublinear}, it holds for all $k\geq 1$ that
\begin{align*}
     \|\bx^{k+1} - \bar x^{k+1} \one^T\|_F^2 &\leq c_1 \rho^{2k} + c_2 \sum_{i=1}^{k} \rho^{k-i} \alpha_i^2\|\bx^i-\bar{x}^i\one^T\|_F^2 + c_3\bar{\phi}\sum_{i=1}^{k} \rho^{k-i} \alpha_i^2\\
     &\leq  c_1 \rho^{2k} + (c_2+c_3){\bar{\sG}}\sum_{i=1}^{k} \rho^{k-i} \alpha_i^2 \leq {\sG} \alpha_{k+1}^2,\quad \text{for some $\sG>0$}.
\end{align*}
The existence of such constant $\sG$ is ensured by \cref{lem:sublinear} and the fact that $\{\rho^k\}_k$ converges linearly to zero. 
\end{proof}

\subsection{Proof of \cref{prop:dgd}}\label{proof:dgd-descent}
\begin{proof}
Based on \cref{lem:dgd-descent} and consensus error bound in \cref{lemma:dgd}, there exists a constant $\sG>0$ such that   
	\[ f(\bar x^{k+1})  \leq f(\bar x^k) - \frac{\alpha_k}{2}\|\nabla f(\bar x^k)\|^2  + \frac{{\sG\sL}}{2n}\alpha_k^3,\quad \text{for all}\;k\geq 1.\]
Hence, (a) is verified. 
 It follows from the update steps of \cref{algo:dgd} and \cref{as:matrix} (a) that
	\[
	\bar x^k - \bar x^{k+1} = \frac{ \bx^k W\one - \bx^{k+1} \one}{n} =\frac{\alpha_k\nabla F(\bx^k)\one}{n}.
	\]
	Using the Lipschitz continuity of $\nabla f_i$ and \cref{lemma:dgd}, we obtain
\begin{align*}
	\|\bar x^k - \bar x^{k+1}\| &= \alpha_k\left\| \frac{\nabla F(\bx^k)\one}{n}\right\| \leq \alpha_k\|\nabla f(\bar x^k)\| + \alpha_k \left\|\frac{\nabla F(\bx^k)\one}{n} - \nabla f(\bar x^k) \right\|\\
	& = \alpha_k\|\nabla f(\bar x^k)\| + \frac{\alpha_k}{n}\left\|{\sum}_{i=1}^{n} \nabla f_i(\bar x^k) - \nabla f_i(x^k_i) \right\| \leq \alpha_k\|\nabla f(\bar x^k)\| + \frac{\sL\alpha_k}{n}\sum_{i=1}^{n}\| \bar x^k - x^k_i \| \\
	&\leq \alpha_k\|\nabla f(\bar x^k)\| + \frac{\sL\alpha_k}{n}\sqrt{n\,{\sum}_{i=1}^{n}\| \bar x^k - x^k_i \|^2} \leq  \alpha_k\|\nabla f(\bar x^k)\| + \frac{{\sL}\sqrt{\sG}}{\sqrt{n}}\alpha_k^2,\quad \forall\, k\geq 1,
\end{align*}
where the last inequality is due to ${\sum}_{i=1}^{n}\| \bar x^k - x^k_i \|^2= \|\bar x^k \one^T - \bx^k\|^2_F$. This completes the statement (b).
\end{proof}

\subsection{Proof of Supporting Lemmas}
\subsubsection*{Proof of \cref{lem:sublinear}}\label{proof:lem:sublinear}
\begin{proof}
Since the sequence $\{a_k\}_k$ converges to zero sublinearly, for the fixed constant $\eta:=\frac{\rho+1}{2}\in(\rho,1)$, there exists $k_1 \geq 1$ such that $\frac{a_{k+1}}{a_k} \geq \eta$ for all $k\geq k_1$. We have for all $t> k_1$,
	\begin{align*}
		\sum_{i=1}^{t} a_i \rho^{t-i} = \sum_{i=1}^{k_1} a_i \rho^{t-i} + \sum_{i=k_1+1}^{t}a_i \rho^{t-i} \leq \rho^t (\sum_{i=1}^{k_1} a_i \rho^{-i}) +  \frac{a_{t+1}}{\eta} \sum_{i=k_1+1}^{t}(\rho/\eta)^{t-i} \leq \hat{c} \rho^t + \frac{a_{t+1}}{\eta - \rho}, 
	\end{align*}
	where $\hat{c}:=\sum_{i=1}^{k_1} a_i \rho^{-i}$ that is finite and fixed. Our next step is to show $\rho^t=\cO(a_{t+1})$. Using the relation  \[a_{t+1} \geq \eta^{t+1-k_1} a_{k_1} = \frac{a_{k_1}}{\eta^{k_1-1}} ({\eta}/{\rho})^t \rho^t,\quad\text{for all }t \geq k_1\] and noting that $(\eta/\rho)^t\to\infty$ as $t\to\infty$, we conclude that there exists $k_2 \geq k_1$ such that for all $t\geq k_2$,
\[	({\eta}/{\rho})^t \geq  \frac{\hat{c}(\eta-\rho)\eta^{k_1-1}}{a_{k_1} }\quad \Longrightarrow \quad
	a_{t+1} \geq \hat{c}(\eta-\rho)\rho^t \quad  \Longrightarrow \quad \sum_{i=1}^{t} a_i \rho^{t-i} \leq \frac{2 a_{t+1}}{\eta-\rho} = \frac{4\alpha_{t+1}}{1-\rho}.
\]
Finally, setting $\tilde c:=\max\{\max_{1 \leq k < k_2} \frac{\sum_{i=1}^{k} a_i \rho^{k-i} }{a_{k+1}} ,\frac{4}{1-\rho}\}$ completes the proof.
\end{proof}

\subsubsection*{Proof of \cref{lem:dgd-descent}} \label{proof:lem:dgd-descent}
\begin{proof}
		Firstly, we show condition \ref{C1} holds. Based on the update of \cref{algo:dgd}, 
	\begin{equation}
		\label{eq:prop-dgd-est-1}
		f(\bar x^{k+1}) = f\left(\frac{\bx^{k+1} \one}{n}\right) = f\left(\frac{(\bx^k W - \alpha_k \nabla F(\bx^k)) \one}{n}\right) = f\left(\frac{\bx^kW\one}{n}- \alpha_k\cdot  \frac{\nabla F(\bx^k)\one}{n} \right).
	\end{equation}
	It follows from \cref{as:matrix} (a) that
	\begin{equation*}
		%\label{eq:prop-dgd-est-2}
		\bx^k W \one = \sum_{i=1}^{n} \sum_{j=1}^{n}w_{ji} x_j = \sum_{j=1}^{n}x_j\sum_{i=1}^{n}w_{ji} = \bx^k\one.
	\end{equation*}
	Combining this relation with \cref{eq:prop-dgd-est-1} and utilizing the well-known descent lemma \cite[Theorem 2.1.5]{nesterov2018lectures}, we obtain
	\begin{equation}
		\label{eq:prop-dgd-est-3}	
		\begin{aligned}
			f(\bar x^{k+1}) & = f\left(\bar x^k - \alpha_k \cdot \frac{\nabla F(\bx^k)\one}{n}\right) \leq f(\bar x^k) - \alpha_k \cdot \iprod{\nabla f(\bar x^k)}{\frac{\nabla F(\bx^k)\one}{n}} + \frac{{\sf L}\alpha_k^2}{2} \left\|\frac{\nabla F(\bx^k)\one}{n} \right\|^2 \\
			& = f(\bar x^k) - \frac{\alpha_k}{2}\|\nabla f(\bar x^k)\|^2 -  \frac{\alpha_k(1-{\sf L}\alpha_k)}{2}\left\|\frac{\nabla F(\bx^k)\one}{n} \right\|^2 + \frac{\alpha_k}{2}\left\| \nabla f(\bar x^k) - \frac{\nabla F(\bx^k)\one}{n} \right\|^2,
		\end{aligned}
	\end{equation}
	where the last equality is due to $-2\iprod{a}{b}=\|a-b\|^2 -\|a\|^2 - \|b\|^2.$ The ${\sL}$-smoothness of $f_i$ gives
	\begin{equation*}
	%\label{eq:prop-dgd-est-3-1}	
		\left\| \nabla f(\bar x^k) - \frac{\nabla F(\bx^k)\one}{n} \right\|^2=\frac{1}{n^2}\left\|\sum_{i=1}^{n} \nabla f_i(\bar x^k) - \nabla f_i(x^k_i) \right\|^2 \leq \frac{{\sf L^2}}{n}\sum_{i=1}^{n}\|\bar x^k - x_i^k\|^2=\frac{{\sf L^2}}{n} \|\bar x^k \one^T - \bx^k\|^2_F.
	\end{equation*}
The proof is completed by substituting the above estimate into \eqref{eq:prop-dgd-est-3}.
\end{proof}

\section{Appendix: Federated Averaging}
\subsection{Proof of \cref{lemma:FedAvg}}\label{proof:FedAvg}
\begin{proof}
	First, by $\sL$-smoothness of $f$, we have
	\begin{equation}
		\label{eq:fed-1}
		\begin{aligned}
			f(x^{k+1}) &\leq f(x^k) - \iprod{\nabla f(x^k)}{x^k-x^{k+1}} + \frac{{\sL}}{2}\|x^{k+1}-x^k\|^2\\
                & = f(x^k) - \frac{1}{Emn^2}\iprod{Emn\sqrt{\alpha_k}\nabla f(x^k)}{n(x^k-x^{k+1})/\sqrt{\alpha_k}} + \frac{{\sL}}{2}\|x^{k+1}-x^k\|^2\\
			& = f(x^k) - \frac{Em\alpha_k}{2}\|\nabla f(x^k)\|^2 - \frac{1}{2}\left(\frac{1}{Em\alpha_k}-{\sL}\right)\|x^k-x^{k+1}\|^2 \\& \hspace{3ex}+ \frac{\alpha_k}{2Emn^2}\left\| \frac{n(x^k-x^{k+1})}{\alpha_k} - Emn\cdot \nabla f(x^k)\right\|^2,
 		\end{aligned}
	\end{equation} 
where the last line uses $2\iprod{a}{b}=\|a\|^2+\|b\|^2-\|a-b\|^2$ with $a=Emn\sqrt{\alpha_k}\nabla f(x^k)$ and $b=n(x^k-x^{k+1})/\sqrt{\alpha_k}$. To obtain the desired result, we primarily focus on the upper bound of $\|n(x^k-x^{k+1})/\alpha_k - Emn\nabla f(x^k)\|^2$. According to the update of the central server, it holds that
	\begin{equation}
		\label{eq:fed-2}
		n(x^k - x^{k+1}) = \sum_{t=1}^{n} (x^k-x_t^{k+1}).
	\end{equation}
	We examine the computation procedures in each client $t$ and observe that 
	\begin{equation}
		\label{eq:fed-3}
		x^k - x_t^{k+1} = \alpha_k \sum_{i=1}^{E} \sum_{j=1}^{m} \nabla h^t_{\pi^{i,j}_{t}}(y_{t,k}^{i,j}) ,\quad \forall \; 1 \leq t \leq n.
	\end{equation}
Combining \cref{eq:fed-2} and \cref{eq:fed-3} leads to
	\begin{equation}
		\label{eq:fed-4}
		\frac{n(x^k - x^{k+1})}{\alpha_k} = \sum_{t=1}^{n} \sum_{i=1}^{E}   \sum_{j=1}^{m} \nabla h^t_{\pi^{i,j}_{t}}(y_{t,k}^{i,j}).
	\end{equation}
Notice that
$
\nabla f(x) = \frac{1}{n}\sum_{t=1}^{n} \nabla f_t(x) = \frac{1}{E n m} \sum_{t=1}^{n}  \sum_{i=1}^E   \sum_{j=1}^{m}\nabla h^t_{\pi^{i,j}_{t}}(x)
$, then by the triangle inequality, 
\begin{equation}\label{eq:fed-initial-bound}
\begin{aligned}
	 &\hspace{4mm}\left\| \frac{n(x^k-x^{k+1})}{\alpha_k} - Emn \cdot \nabla f(x^k)\right\|  \leq \sum_{t=1}^{n} \sum_{i=1}^{E}   \sum_{j=1}^{m} \|\nabla h^t_{\pi^{i,j}_{t}}(y_{t,k}^{i,j}) - \nabla h^t_{\pi^{i,j}_{t}}(x^k)\|\\
	&\leq \sL \sum_{t=1}^{n} \sum_{i=1}^{E}  \sum_{j=1}^{m} \|y_{t,k}^{i,j} - x^k \|\leq E{\sL}m \sum_{t=1}^{n} \max_{i\in[E],j\in[m]} \|y_{t,k}^{i,j} - x^k \|   =: V_k, 
\end{aligned}
\end{equation}
where the last line is due to Lipschitz continuity of $\nabla h^t_j(\cdot)$. Based on the update of $\Fed$, we unfold $\|y_{t,k}^{i,j} - x^k\|$ and invoke the triangle inequality again, then the following holds for all client $t\in[n]$,
\[
\|y_{t,k}^{i,j} - x^k\| \leq \alpha_k \sum_{\ell=1}^{E} \sum_{r=1}^{m}\|\nabla h^t_{\pi_t^{\ell,r}}(y_{t,k}^{\ell,r})\|,\quad \forall\,  i\in[E]\; \text{and}\; \forall\, j\in[m].
\]
Based on the definition of $V_k$ in \eqref{eq:fed-initial-bound} and the above estimate, we can establish
\begin{equation}\label{eq:fed-bound-v-1}
	\begin{aligned}
	V_k & \leq E\sL m\alpha_k \sum_{t=1}^{n} \sum_{\ell=1}^{E} \sum_{r=1}^{m}\|\nabla h^t_{\pi_t^{\ell,r}}(y_{t,k}^{\ell,r})\| \\
	& \leq E\sL m\alpha_k \sum_{t=1}^{n} \sum_{\ell=1}^{E} \sum_{r=1}^{m} \big(\|\nabla h^t_{\pi_t^{\ell,r}}(y_{t,k}^{\ell,r})-\nabla h^t_{\pi_t^{\ell,r}}(x^k)\| + \|\nabla h^t_{\pi_t^{\ell,r}}(x^k)\|\big) \\
	& \leq E\sL m\alpha_k V_k + E\sL m\alpha_k \sum_{t=1}^{n} \sum_{\ell=1}^{E} \sum_{r=1}^{m}  \|\nabla h^t_{\pi_t^{\ell,r}}(x^k)\| = E\sL m\alpha_k V_k + E^2\sL m\alpha_k \sum_{t=1}^{n} \sum_{r=1}^{m}  \|\nabla h^t_{r}(x^k)\|,
\end{aligned}
\end{equation}
where the second line is due to triangle inequality and the last equation holds because $\{\pi_t^{\ell,1},\ldots,\pi_t^{\ell,m}\}$ is a permutation of $[m]$, indicating $\sum_{r=1}^{m}\|\nabla h^t_{\pi_t^{\ell,r}}(x^k)\|$ = $\sum_{r=1}^{m}\|\nabla h^t_{r}(x^k)\|$.
Note that $h^t_{r}(\cdot)$ is $\sL$-Lipschitz smooth and lower bounded by $\bar f$ for all $t\in[n]$ and $r\in[m]$, utilizing \eqref{eq:L-smooth-Fed},  we have 
\[
\|\nabla h^t_{r}(x^k)\|^2 \leq  2\sL(h^t_{r}(x^k) - \bar f) \leq 2mn\sL(f(x^k) -\bar f)\quad \Longrightarrow \quad   \|\nabla h^t_{r}(x^k)\| \leq \sqrt{2mn\sL(f(x^k) -\bar f)}.
\]
Since $\alpha_k \leq \frac{1}{2mE\sL}$, we rearrange \eqref{eq:fed-bound-v-1} and apply the bound shown above, then the quantity $V_k$ is upper bounded by
\begin{equation}\label{eq:fed-important-0}
	V_k \leq 2E^2\sL m\alpha_k \sum_{t=1}^{n} \sum_{r=1}^{m}  \|\nabla h^t_{r}(x^k)\| \leq 2E^2\sL m^2 n \alpha_k \sqrt{2mn\sL(f(x^k) -\bar f)}. %2\alpha_k \sqrt{2\sL N(f(x^k) - \bar f)}\sum_{t=1}^{T} m_t \leq 2\alpha_k N\sqrt{2\sL N(f(x^k) - \bar f)}.
\end{equation}
Using \eqref{eq:fed-initial-bound} and inserting this upper bound of $V_k$ into \eqref{eq:fed-initial-bound}, we obtain
\begin{equation}
	\label{eq:fed-grad-err}
	\left\| \frac{n(x^k-x^{k+1})}{\alpha_k} - Emn \cdot \nabla f(x^k)\right\|^2 \leq V_k^2 \leq 8E^4\sL^3 m^5n^3 (f(x^k) - \bar f)\alpha_k^2.
\end{equation} 
It then follows from estimates \eqref{eq:fed-1} and \eqref{eq:fed-grad-err} and the condition $\alpha_k \leq \frac{1}{2mE\sL}$ that
\begin{equation}
	\label{eq:fed-important}
	\begin{aligned}
f(x^{k+1}) - \bar f  &\leq  (f(x^k) - \bar f) - \frac{Em\alpha_k}{2}\|\nabla f(x^k)\|^2 + \frac{\alpha_k}{2Emn^2}\left\| \frac{n(x^k-x^{k+1})}{\alpha_k} - Emn\cdot \nabla f(x^k)\right\|^2\\ &\leq (1+ \underbracket[.75pt][4pt]{4E^3\sL^3 m^4 n}_{=:\sH} \alpha_k^3)(f(x^k)-\bar f) - \frac{Em\alpha_k}{2}\|\nabla f(x^k)\|^2.
\end{aligned}
\end{equation}
Since $\sum_{k=1}^\infty \alpha_k^3<\infty$ and it holds for all $k\geq 1$ that
\[f(x^{k+1}) \leq (1+\sH\alpha_k^3)(f(x^k)-\bar f) \leq (f(x^1)-\bar f) \cdot {\prod}_{i=1}^k (1+\sH\alpha_i^3) \leq  (f(x^1)-\bar f) \cdot \exp\Big(\sH{\sum}_{i=1}^k \alpha_i^3\Big) ,\] we conclude that the sequence $f(x^k)$ is bounded above for all $k\geq 1$. Thus, there exists a constant $\sG>0$ such that $\sH(f(x^k) - \bar f) \leq \sG$ for all $k\geq1$. It follows from \eqref{eq:fed-important-0} and \eqref{eq:fed-important} that 
\begin{align*}
&	f(x^{k+1})   \leq f(x^k) - \frac{Em\alpha_k}{2}\|\nabla f(x^k)\|^2 +   \sG  \alpha_k^3,\quad \text{and}\\
&	\|x^k-x^{k+1}\| \leq Em  \alpha_k \|\nabla f(x^k)\| + \frac{\alpha_k}{n}\left\| \frac{n(x^k-x^{k+1})}{\alpha_k} - Emn\nabla f(x^k)\right\|   \\
&\hspace{12ex}\leq  Em \alpha_k \|\nabla f(x^k)\| + \sqrt{2mE\sG} \cdot \alpha_k^2,
\end{align*}
as desired.
\end{proof}
\end{document}